\documentclass[11pt,reqno]{amsart}%
\usepackage[a4paper,margin=3cm]{geometry}%
\usepackage[pagebackref]{hyperref}%
\usepackage[T1]{fontenc}%
\usepackage[utf8]{inputenc}%
\usepackage{lmodern,mathtools}
\usepackage{amsmath,amsfonts,amssymb,amsthm}
\usepackage{amsmath}
\usepackage{tikz}
\usepackage{dsfont}

\title[A new approach for the unitary Dyson Brownian motion]{A new approach for the unitary Dyson Brownian motion through the theory of viscosity solutions}

\author{Charles Bertucci}
\address[CB]{CMAP, Ecole Polytechnique, UMR 7641, 91120 Palaiseau, France.}
\email{\url{mailto:charles.bertucci@polytechnique.edu}}
\author{Valentin Pesce}
\address[VP]{CMAP, Ecole Polytechnique, UMR 7641, 91120 Palaiseau, France.}
\email{\url{mailto:valentin.pesce@polytechnique.edu}}

\date{Spring 2025, compiled \today}
\newtheorem{theorem}{Theorem}[section]%
\newtheorem{definition}[theorem]{Definition}
\newtheorem{corollary}[theorem]{Corollary}%
\newtheorem{proposition}[theorem]{Proposition}%
\newtheorem{lemma}[theorem]{Lemma}%
\newtheorem{remark}[theorem]{Remark}%

\let\ensembleNombre\mathbb
\newcommand\tab[1][0.5cm]{\hspace*{#1}}

\newcommand\mesT{\mathcal P(\mathbb T)}
\newcommand\N{\ensembleNombre{N}}
\newcommand\Z{\ensembleNombre{Z}}

\newcommand\R{\ensembleNombre{R}}
\newcommand\C{\ensembleNombre{C}}
\newcommand\T{\ensembleNombre{T}}
\newcommand\dis{\displaystyle}
\newcommand\E{\mathbb{E}}
\newcommand\mP{\mathbb{P}}

\newcommand\equalaw{\underset{law}{=}}
 
\DeclareMathOperator{\leb}{Leb} 
\DeclareMathOperator{\cotan}{cotan} 

\makeatletter
\def\@MRExtract#1 #2!{#1}     
\renewcommand{\MR}[1]{
  \xdef\@MRSTRIP{\@MRExtract#1 !}%
  \href{http://www.ams.org/mathscinet-getitem?mr=\@MRSTRIP}{MR-\@MRSTRIP}}
\makeatother

\numberwithin{equation}{section}


\begin{document}

\begin{abstract}
In this paper, we study the unitary Dyson Brownian motion through a partial differential equation approach recently introduced for the real Dyson case. 
The main difference with the real Dyson case is that the spectrum is now on the circle and not on the real line, which leads to particular attention to comparison principles. 
First we recall why the system of particles which are the eigenvalues of unitary Dyson Brownian motion is well posed thanks to a containment function. 
Then we proved that the primitive of the limit spectral measure of the unitary Dyson Brownian motion is the unique solution to a viscosity equation obtained by primitive the Dyson equation on the circle. 
Finally, we study some properties of solutions of Dyson's equation on the circle. We prove a $L^\infty$ regularization. We also look at the long time behaviour in law of a solution through a study of the so-called free entropy of the system. We conclude by discussing the uniform convergence towards the uniform measure on the circle of a solution of the Dyson equation.
\end{abstract}

\maketitle
 
{\footnotesize\tableofcontents}

\section{Introduction}
In this paper we study the so-called unitary Dyson Brownian motion. In \cite{dyson1962}, Dyson derived a system of stochastic differential equations satisfied by the angles of the eigenvalues of the unitary Dyson Brownian motion. We prove that this system of particles in interaction is well defined as stated in \cite{cepa2001brownian} by using the usual method to study systems of particles with a singular potential interaction as in \cite{bolley2018,cepa2001brownian}. 
\newline
\tab Then we look at the the limit equation, when the size of the matrix $N$ goes to $+\infty$, of the spectral measure of the unitary Dyson Brownian motion, which is called Dyson's equation on the circle. In the literature, the convergence of this spectral measure is obtained by showing that this sequence is relatively compact for a certain topology and then by showing uniqueness to Dyson's equation on the circle by looking at the moments or the Stieltjes transform of a solution to Dyson's equation in a certain sense that we do not define here \cite{cepa2001brownian}. As noticed in \cite{bertucci2022spectral}, for the real Dyson case we can have a more general approach by looking at the partial differential equation satisfied by the cumulative distribution function of a solution to Dyson's equation. Thanks to this approach we can deal with more general drifts for which the computation of the exact solution would not be possible a priori \cite{bertucci2022spectral}. We follow the same idea for Dyson's equation on the circle. We first integrate this equation to obtain what we shall call the primitive of the Dyson equation. We pay attention to the fact that in the case of measures on the circle there is no canonical cumulative distribution function associated to a measure on the circle as in the real case. The resulting equation is an integro-differential partial differential equation. This equation involves the half Laplacian on the circle which satisfies a maximum principle which leads us to define a notion of viscosity solution for the primitive equation as for the real case. We obtain a comparison principle and uniqueness of viscosity solution for the primitive equation by using classical techniques for viscosity equations (\cite{crandall1992} for second order fully non-linear partial differential equation and \cite{barles2008,arisawa2006,arisawa2008} for integro-differential partial differential equation for instance). 
\newline
\tab Later we go back to the system of particles associated to the unitary Dyson Brownian motion to prove a result of convergence of this systems of particles. As in \cite{bertucci2022spectral}, a main argument in the proof is a discrete comparison result of two system of particles on the circle. However there is no "good" order on the circle. This creates a difficulty to state a comparison principle. To overcome this difficulty, we consider a periodized system of particles associated to the initial system of particles. Thanks to this, we can state properly a discrete comparison principle. In the same spirit as in \cite{bertucci2022spectral}, we prove the convergence of the primitive of the spectral measure of the unitary Dyson Brownian motion to the unique viscosity solution to the primitive Dyson equation.
The proof mostly uses the discrete comparison principle and the maximum property of the half Laplacian. However, the main difference with the real case is that in the proof we have to use the local formulation of viscosity solutions. Indeed, locally an arc of the circle is like an interval of $\R$ and so there is a natural order on it which allows us to use the discrete comparison principle. 
\newline
\tab Finally, we study properties of the Dyson equation on the circle. We prove a $L^\infty$ regularization of the Dyson flow by proving an a priori estimate for a smooth solution to Dyson's equation on the circle and then using the vanishing method for viscosity solutions of the primitive of the Dyson equation as in \cite{bertucci2024spectral} for the real Dyson case. We give some other properties as the decay of the $L^p$ norms and the fact that the so-called free entropy associated to Dyson's flow is monotone and continuous along Dyson's flow. Thanks to these properties we give a new proof of the convergence in law of a solution to Dyson's equation towards the uniform measure on the circle as stated in \cite{cepa2001brownian}. We also study the convergence to the equilibrium of the solution to Dyson's equation. We shall prove that for all $p\in[1,+\infty)$ the solution converges towards the uniform measure on the circle in $L^p$ and if the solution is positive it also converges in $L^{\infty}$.

\section{The Dyson model: from random matrices to particles in interaction}
\label{se:matrix to part}

\subsection{The Dyson model}In 1962 \cite{dyson1962}, Dyson introduced continuous models that defined new types of Coulomb gases through the spectrum of random matrices. This approach gives another point of view for well known theorems in random matrix theory as the Wigner theorem \cite{anderson2010,meckes2019random,tao2012}. It makes a connection between random matrices and systems of particles in interaction and opened many links between random matrices and Coulomb gases, see for instance  \cite{bolley2018,chafai2015,chafai2021aspects,serfaty2024lectures}. 
Here we focus on the so-called unitary Dyson Brownian motion. 
Let us recall that for an integer $N\ge 1$, we say that a matrix $M$ is sampled from the Gaussian unitary ensemble of size $N$ ($GUE_N$) if $M$ is an hermitian matrix of size $N\times N$ such that $(M_{i,j})_{1\le i<j\le N}$ are independent complex normal law of mean 0 and variance 1 and $(M_{i,i})_{1\le i\le N}$ are independent (and independent of $M_{i,j}$ for $1\le i<j\le N$) real normal law of mean 0 and variance 1. The law of a matrix sampled from it is invariant by unitary change of basis. 
\newline
For all integer $N\ge 1$, let $\mathcal U_N(\C)$ be the set of complex unitary matrices. For $N\ge1$, we define the unitary Dyson Brownian motion as a stochastic process starting from a deterministic unitary matrix $D_N(0)\in \mathcal U_N(\C)$ and generated by: 
\begin{equation}
D_N(t+h)=D_N(t)\exp\left(i\cfrac{\sqrt h}{\sqrt{N}}M \right)\in\mathcal U_N(\C)
\label{Genmatrix}
\end{equation} where $M$ is sampled from the Gaussian unitary ensemble of size $N$ (re sampled at each step).
This definition gives us an algorithm to simulate the  unitary Dyson Brownian motion \cite{buijsman2024}.
\begin{remark}
This definition generalizes in dimension $N$ the fact that if $(B_t)_{t\ge0}$ is a real Brownian motion we have: $$\exp(iB_{t+h})\equalaw\exp(iB_t)\exp(i \sqrt{h}\mathcal N),$$
with $\mathcal N$ a normal law of parameter (0,1) independent from $B_t$ thanks to the fact that $(B_{t+h}-B_t)_{h\ge0}$ has the law of a Brownian motion independent of $B_t$.
\end{remark} 
Let $(\mu_t^j)_{1\le j\le N}=(\exp(i\theta_t^j))_{1\le j\le N}$ be the $N$ eigenvalues of $D_N(t)$ for $t\ge 0$. 
Using the perturbation theory (or Hadamard's variational formula (\cite{tao2012},section 1.3.4)), Dyson explained that the family $(\theta_.^j)_{1\le j\le N}$ are weak solutions of a system of stochastic differential equations (SDE). 
\newline
From now until the end, for $N\ge 1$, we say that a system of $N$ particles; $(\lambda_t^i)_{1\le i\le N,t\ge 0}$ satisfies the unitary Dyson Brownian system of size $N$ (sometimes we just call it $N$ particles Dyson case) if it satisfies the following system of SDE:
\begin{equation} 
\forall 1\le i\le N,\, d\lambda_t^i=\cfrac{1}{N}\dis\sum_{j\ne i}\cotan((\lambda_t^i-\lambda_t^j)/2)\,dt+\cfrac{2}{\sqrt{N}}\,dB_t^i,
\label{dysononcirc}
\end{equation}
with $(B_t^i)_{1\le i\le N}$ a family of $N$ independent Brownian motions (defined on the same probability space and a same filtration).
\newline
This system contains a diffusive part through the Brownian motion term in competition with a repulsive mean field interaction through the first sum (each particle is repulsed by all the other particles). Hence the existence of solutions in a weak or strong sense of the previous system is not clear a priori. 
\newline
Even though results for existence and uniqueness of a weak and strong solution for this system are already known in the literature \cite{anderson2010,bolley2018,cepa2001brownian,cepa2007}, complete proofs are missing and we shall give in the following sections a detailed proof of these points for the sake of completeness.

\subsection{Notation and well-posedness of the model}

\subsubsection{Notation} The open subset of $\R^N$ on which we work is:  $$D=\R^N-\cup_{i\ne j}\{(x_1,...,x_N)\in\R^N: x_i=x_j\text{ modulo } 2\pi\}$$ and its boundary in the compactification of $\R^N$:  $$\partial D=\{+\infty\}\cup_{i\ne j}\{(x_1,...,x_N)\in\R^N: x_i=x_j \text{ modulo } 2\pi\}.$$
Let us define $\R_>^N$ by: $\{(x_1,...,x_N)\in\R^N: x_1<x_2<...<x_N<x_1+2\pi\}$.
\newline
Furthermore, we shall write $x=(x_1,...,x_N)$ for an element of $\R^N$. 
\newline 
If $x\in D$, let us define its energy by $$\mathcal E(x):=\cfrac{1}{N}\dis\sum_{i=1}^NV(x_i)+\cfrac{1}{N^2}\dis\sum_{1\le i \ne j \le N} W(x_i-x_j)=:\mathcal E_V(x)+\mathcal E_W(x)$$ where $V$ is an external confinement potential which is in our case $V(t):=t^2$ as the kinetic energy that help us to check that our system does not explode spatially in finite time and for $x\in\R$, $x\ne 0$ modulo $2\pi$, $W(x):=-\ln((\sin(x/2)^2)\ge 0$ which is the potential energy associated to our system since $-W'(x)=\cotan(x/2)$.
\newline
Let $\alpha_N$ and $\beta_N$ be two non negative real numbers. 
We are interested in the following SDE system satisfied by a system of particles $(\lambda_.^i)_{1\le i\le N}$:
\begin{equation} 
\begin{split}
\forall 1\le i\le N,\, d\lambda_t^i=&-\cfrac{\alpha_N}{2}\cfrac{\partial \mathcal E_W}{\partial x_i}(\lambda_t^1,...,\lambda_t^N)+\sqrt{\cfrac{2\alpha_N}{\beta_N}}\,dB_t^i\\
=&-\cfrac{\alpha_N}{N^2}\dis\sum_{j\ne i}W'(\lambda_t^i-\lambda_t^j)\,dt+\sqrt{\cfrac{2\alpha_N}{\beta_N}}\,dB_t^i\\
=&\cfrac{\alpha_N}{N^2}\dis\sum_{j\ne i}\cotan((\lambda_t^i-\lambda_t^j)/2)\,dt+\sqrt{\cfrac{2\alpha_N}{\beta_N}}\,dB_t^i
\end{split}
\label{dysoncircleeq}
\end{equation}
with $(B_t^i)_{1\le i\le N}$ a family of independent Brownian motions.
\newline
The unitary Dyson Brownian system corresponds to the case:  $\alpha_N=N$ and $\beta_N=N^2/2$.

\subsubsection{Formulas}
\label{se:formulas}
The goal of this section is to compute how the generator of the system of particles \eqref{dysoncircleeq} acts on $\mathcal E$.
\newline
Firstly the generator associated to the system $(\lambda_t^i)_{1\le i\le N}$ solution to $\eqref{dysoncircleeq} $ is: $$Lf=\cfrac{\alpha_N}{\beta_N}\Delta f-\cfrac{\alpha_N}{2}\nabla \mathcal E_W.\nabla f,$$ for $f$ a smooth function from $D$ to $\R$. Hence for a smooth function $f$ from $D$ to $\R$, by Itô's formula: $$\E(f(\lambda_t^1,..,\lambda_t^N))=f(\lambda_0^i,...,\lambda_0^N)+\E\left(\dis\int_0^t Lf(\lambda_t^1,...,\lambda_t^N)dt\right).$$
\newline
Let $x\in D$ and $1\le i,j \le N$.
\newline
We have the following immediate formulas: 
\begin{align}
\cfrac{\partial \mathcal E_V}{\partial x_i}(x)&=\cfrac{2x_i}{N}\\
\cfrac{\partial \mathcal E_W}{\partial x_i}(x)&=-\cfrac{2}{N^2}\dis\sum_{j\ne i}\cotan((x_i-x_j)/2)\\
\nabla^2 \mathcal E_V(x)&=\cfrac{2}{N}I_N\\
\cfrac{\partial^2 \mathcal E_W}{\partial x_i^2}(x)&=\cfrac{1}{N^2}\sum_{j\ne i}\cfrac{1}{\sin((x_i-x_j)/2)^2}
\end{align}
Hence, we can now compute $L\mathcal E_V$ and $L\mathcal E_W$ to compute finally $L\mathcal E$.
\newline
Remark that $$L\mathcal E_V(x)=\cfrac{\alpha_N}{\beta_N}\times 2-\cfrac{\alpha_N}{2}\dis\sum_{1\le i\ne j\le N}\cfrac{-4x_i}{N^3}\cotan((x_i-x_j)/2)=\cfrac{2\alpha_N}{\beta_N}+\cfrac{2\alpha_N}{N^3}\dis\sum_{1\le i\ne j\le N}x_i\cotan((x_i-x_j)/2).$$
To express the last sum we can notice that: $$\dis\sum_{1\le i\ne j\le N}x_i\cotan((x_i-x_j)/2)=\dis\sum_{1\le i\ne j\le N}x_j\cotan((x_j-x_i)/2)$$
So we get: $$2\dis\sum_{1\le i\ne j\le N}x_i\cotan((x_i-x_j)/2)=\dis\sum_{1\le i\ne j\le N}(x_i-x_j)\cotan((x_i-x_j)/2).$$
Hence, we have the formula: $$L\mathcal E_V(x)=\cfrac{2\alpha_N}{\beta_N}+\cfrac{2\alpha_N}{N^3}\sum_{1\le i\ne j\le N}\psi(x_i-x_j),$$ where $\psi(x)=\cfrac{x}{2}\cotan(x/2)$. Let us remark that this function is such that $\psi(x)\le 1$ for all $x\in (-2\pi,2\pi)$.
\newline
Now we compute $L\mathcal E_W(x)$.
Start with: 
\begin{align}
L\mathcal E_W(x)&=\cfrac{\alpha_N}{\beta_N}\dis\sum_{i=1}^N\cfrac{1}{N^2}\dis\sum_{j\ne i}\cfrac{1}{(\sin((x_i-x_j)/2)^2}-\cfrac{\alpha_N}{2}|\nabla \mathcal E_W(x)|^2.
\end{align}
We can express $|\nabla \mathcal E_W(x)|^2$: 
\begin{align*}
|\nabla \mathcal E_W(x)|^2=&\cfrac{4}{N^4}\sum_{1\le i\le N}\sum_{k\ne i,\, j\ne i}\cotan((x_i-x_k)/2)\cotan((x_i-x_j)/2)\\
=&\cfrac{4}{N^4}\left[\sum_{j\ne i, k \ne i, j\ne k}\cotan((x_i-x_k)/2)\cotan((x_i-x_j)/2)+\sum_{i\ne j}(\cotan((x_i-x_j)/2))^2\right]\\
=:&\cfrac{4}{N^4}\left[S_N+\sum_{i\ne j}(\cotan((x_i-x_j)/2))^2\right]
\end{align*} 
by separating the second sum when $j=k$ or $j\ne k$.
Now to compute $S_N$ we need to use an algebraic identity. This computation is based on the fact that if $a+b+c=0$ then: 
\begin{equation}
\label{eqcotan}
\cotan(a)\cotan(b)+\cotan(b)\cotan(c)+\cotan(a)\cotan(c)=1
\end{equation}
Thanks to this remark, 
\begin{align*}
S_N=&\sum_{i\ne k, i\ne j, j\ne k}\cotan((x_i-x_k)/2)\cotan((x_i-x_j)/2)\\
=&-\sum_{i\ne k, i\ne j, j\ne k}\cotan((x_i-x_k)/2)\cotan((x_j-x_i)/2)\\
\underset{\eqref{eqcotan}}{=}&-\sum_{i\ne k, i\ne j, j\ne k}\left(1-\cotan((x_i-x_k)/2)\cotan((x_k-x_j)/2)-\cotan((x_j-x_i)/2)\cotan((x_k-x_j)/2)\right)\\
=&-N(N-1)(N-2)\\
&-\sum_{i\ne k, i\ne j, j\ne k}\left(\cotan((x_k-x_i)/2)\cotan((x_k-x_j)/2)-\cotan((x_j-x_i)/2)\cotan((x_j-x_k)/2)\right)\\
=&-N(N-1)(N-2)-2S_N.
\end{align*}
Hence we have: 
\begin{equation}
S_N=\dis\sum_{j\ne i,k\ne i, j\ne k}\cotan((x_i-x_k)/2)\cotan((x_i-x_j)/2)=-N(N-1)(N-2)/3.
\label{cotlardcircle}
\end{equation}
Moreover by using that $\cotan^2=-1+1/\sin^2$ we get: 
$$\sum_{i\ne j}(\cotan((x_i-x_j)/2))^2=\sum_{i\ne j}\frac{1}{(\sin((x_i-x_j)/2))^2}-N(N-1).$$
So we get: $$|\nabla \mathcal E_W(x)|^2=\cfrac{4}{N^4}\left(-\cfrac{N(N-1)(N-2)}{3}-N(N-1)\right)+\cfrac{4}{N^4}\sum_{i\ne j}\frac{1}{(\sin((x_i-x_j)/2))^2}.$$
This can be summarized into: $$L\mathcal E_W(x)=\left[\cfrac{\alpha_N}{N^2\beta_N}-\cfrac{2\alpha_N}{N^4}\right]\dis\sum_{i=1}^N\dis\sum_{j\ne i}\cfrac{1}{(\sin((x_i-x_j)/2))^2}+C_N, $$ where $C_N:=\cfrac{2\alpha_N}{N^4}\left(\cfrac{N(N-1)(N-2)}{3}+N(N-1)\right)$.
\newline
So we can now compute: $$L\mathcal E(x)=\cfrac{2\alpha_N}{\beta_N}+\cfrac{2\alpha_N}{N^3}\sum_{1\le i\ne j\le N}\psi(x_i-x_j)+\left[\cfrac{\alpha_N}{N^2\beta_N}-\cfrac{2\alpha_N}{N^4}\right]\dis\sum_{i=1}^N\dis\sum_{j\ne i}\cfrac{1}{(\sin((x_i-x_j)/2))^2}+C_N.$$
If $\beta_N\ge N^2/2$ then $\left[\cfrac{\alpha_N}{N^2\beta_N}-\cfrac{2\alpha_N}{N^4}\right]\le0$.
Using that $\psi(x)\le 1$ if $|x|<2\pi$ we deduce that for all $x\in \R^N_>$: 
\begin{equation}
\label{unfiorm bound}
\sup_{x\in\R^N_>} L\mathcal E(x)\le \cfrac{2\alpha_N}{\beta_N}+\cfrac{2\alpha_NN(N-1)}{N^3}+C_N. 
\end{equation}
Let us notice that the Dyson case ($\alpha_N=N$, $\beta_N=N^2/2)$ corresponds to the case $\left[\cfrac{\alpha_N}{N^2\beta_N}-\cfrac{2\alpha_N}{N^4}\right]=0$. 
\begin{remark}
We notice that a key point in the proof is the relation $\eqref{eqcotan}$ satisfied by $\cotan$. Actually $\coth$ satisfies the same relation with -1 instead of 1, and the inverse function satisfies it with 0 instead of 1 (which is used in the existence in real Dyson case where the interaction is the inverse function instead of $\cotan$ \cite{anderson2010,bolley2018}).
\end{remark}

\subsubsection{Condition for well-posedness of the model}
We introduce the following stopping times: for all $\varepsilon>0$, $$T_\varepsilon=\inf\{t\ge0: \underset{1\le i\le N}{\max} |\lambda_t^i|\ge\varepsilon^{-1} \, \text{or}\, \underset{1\le i,j\le N+1}{\min}|\lambda_t^i-\lambda_t^j|\le \varepsilon\}$$ and $$T_{\partial D}=\underset{\varepsilon>0}{\sup}T_\varepsilon,$$ which are the exit time of compact set of $D$.
\newline
One can show that for a good choice of parameters $\alpha_N$ and $\beta_N$ there is no explosion in finite time. 
\begin{theorem}[Criterion for absence of explosion]\label{th:explo}
For $N\ge 1$, if $\beta_N\ge N^2/2$, then 
for any initial data $\lambda_0\in \R^N_>$, there exists a unique strong solution to $~\eqref{dysoncircleeq}$ defined on $[0,+\infty)$ and $T_{\partial D}=+\infty$ almost surely. 
\end{theorem}
This result is known in the literature and the condition on $\beta_N$ is necessary to be sure that there is no explosion \cite{cepa2001brownian}. 
The proof uses the fact that $\mathcal E$ is a good function of containment for the system (i.e $\mathcal E(x)$ goes to $\infty$ when $x$ goes to $\partial D$ and $L\mathcal E$ is bounded uniformly on $\R^N_>$) and the arguments are classical for such system of particles as the real Dyson case or Coulomb gases \cite{anderson2010,bolley2018}.

\begin{proof}
First we state a lemma about the function $\mathcal E$.
\begin{lemma}
\label{lemma: explo}
We have the following properties for $\mathcal E$: 
\begin{enumerate}
\item{$\mathcal E\ge0$}
\item{$\underset{x\to\partial D}{\lim}\mathcal E(x)=+\infty$}
\end{enumerate}
\end{lemma}
\begin{proof}
$\mathcal E_V$ and $\mathcal E_W$ are non negative since $V$ and $W$ are. Hence $\mathcal E$ is non negative on $D$. Moreover we see that if $x$ converges to $\partial D$ there is two cases. First if $||x||$ converges to $\infty$ then $\mathcal E_V(x)$ converges to $+\infty $ and as $\mathcal E_W$ is non negative, $\mathcal E(x)$ converges to $+\infty$. Secondly if $x$ converges to a $y$ such that there exists $i\ne j$ such that $y_i=y_j$ modulo $2\pi$ we clearly get that $\mathcal E_W(x)$ converges to $+\infty$ since $-\ln((\sin(z/2))^2)$ converges to $+\infty$ if $z$ converges to a real of $2\pi \Z$ and as $\mathcal E_V$ is non negative we also have that $H(x)$ converges to $+\infty$. 
\end{proof}
We now prove the theorem.
This proof is quite standard and uses a smooth approximation of the potential and uniqueness and existence for this smooth approximation.
\newline
In the following, if $(\lambda_.^i)_{1\le i\le N}$ is defined by $\eqref{dysoncircleeq}$ we set for all $t$ until this system exists, $\lambda_t^{N+1}=\lambda_t^1+2\pi$.
\newline
Let $\lambda_0\in \R^N_>$ an initial data and $\varepsilon>0 $ smaller than $\min_{1\le i\ne j\le N+1}(|\lambda_0^i-\lambda_0^j|)$ and such that $\varepsilon^{-1}>\max_{i}\lambda_0^i$. 
\newline
We consider a smooth $2\pi$ periodic approximation $W_\varepsilon$ of $W$ which coincides with $W$ on $(\varepsilon, 2\pi-\varepsilon)$ and we define as before $\mathcal E_\varepsilon=\mathcal E_V+\mathcal E_{W_\varepsilon}$. 
By existence and uniqueness of the solution to the SDE we can consider $(\lambda_t^\varepsilon)_{t\ge0}$ starting from $\lambda_0$ solution to the following problem: $$d\lambda_t^\varepsilon=2\sqrt{\cfrac{\alpha_N}{\beta_N}}dB_t-\cfrac{\alpha_N}{2}\nabla \mathcal E_{W_\varepsilon}(\lambda_t^\varepsilon)dt.$$
For $\varepsilon'<\varepsilon$, let $$T_{\varepsilon, \varepsilon'}=\inf\{s>0: \min_{1\le i\ne j\le N+1}|(\lambda_s^{\varepsilon'})^i-(\lambda_s^{\varepsilon'})^j|\le\varepsilon\text{ or }\max_{i}|(\lambda_s^{\varepsilon'})^i|\ge\varepsilon^{-1}\}$$ be a stopping time such that by definition we have that the processes $\lambda^\varepsilon$ and $\lambda^{\varepsilon'}$ coincide up to this time.
Hence, this allows us to define for all $\varepsilon$ as in the beginning of the proof, a stopping time $T_\varepsilon:=T_{\varepsilon,\varepsilon'}$ for $\varepsilon'<\varepsilon$ and $(\lambda_t)_{0\le t\le T_\varepsilon}:=(\lambda_t^{\varepsilon'})_{0\le t\le T_\varepsilon}$ for $\varepsilon'<\varepsilon$. We clearly have that if $\varepsilon'<\varepsilon$, $T_{\varepsilon'}\ge T_{\varepsilon}$ almost surely, and so that we can uniquely defined $\lambda$ until the stopping time $T_{\partial D}$. Moreover by definition of $W_\varepsilon$, we have that $\lambda$ satisfies the SDE with $W$ until the time $T_{\partial D}$ since it satisfies it for all $\varepsilon$ on $[0,T_\varepsilon]$. So it remains to prove that $T_{\partial D}=+\infty$ almost surely to conclude. 
\newline
Let us define for all $R>0$ the stopping times: $$T'_R=\inf\{t\ge0\,:\, \mathcal E(\lambda_t)>R\}$$ and $$T'=\lim_{R\to\infty} T'_R.$$
By the previous lemma: $T'=+\infty\subset T_{\partial D}=+\infty$ because if $\lambda_t$ converges to $\partial D$ in a finite time, $\mathcal E(\lambda_t)$ would explode in finite time by Lemma \ref{lemma: explo}. 
\newline
Let $R>0$, thanks to Lemma \ref{lemma: explo} we can find $\varepsilon$ (that depends on $R$) such that $T_\varepsilon\ge T'_R$.
\newline
Hence, we can use Itô's formula to have for all $t>0$: 
$$\E(\mathcal E_\varepsilon(\lambda^\varepsilon_{t\wedge T'_R}))-\mathcal E_\varepsilon(\lambda_0^\varepsilon)=\E\left(\int_0^{t\wedge T'_R}L\mathcal E_\varepsilon(\lambda_s^\varepsilon)ds\right).$$
Since $T_\varepsilon \ge T'_R$ it gives:
\begin{equation}
\label{Ito formula}
\E(\mathcal E(\lambda_{t\wedge T'_R}))-\mathcal E(\lambda_0)=\E\left(\int_0^{t\wedge T'_R}L\mathcal E(\lambda_s)ds\right).
\end{equation}
Now, by $\eqref{unfiorm bound}$ obtained in the Section $\ref{se:formulas}$, $L\mathcal E$ is bounded on $\R^N_>$ by a constant (which depends on $N$) that we will denote $K_N$. 
\newline 
So we have the following upper bound for every $t>0$, $$\E\left(\int_0^{t\wedge T'_R}L\mathcal E(\lambda_s)ds\right)\le K_Nt.$$ 
As a consequence of the Itô formula $\eqref{Ito formula}$, for every $t>0$, $$\underset{R>0}{\sup}\,\E(\mathcal E(\lambda_{t\wedge T'_R}))<+\infty.$$
Thanks to that we can use a kind of Markov estimate. For every $t>0$ since $\mathcal E$ is non negative we have: $$R \mathds{1}_{T'_R\le t}\le \mathcal E(\lambda_{t\wedge T'_R}).$$
By taking the expectation for every $t>0$ we get: $$\mP(T'_R\le t)\le\cfrac{\underset{R>0}{\sup}\,\E(\mathcal E(\lambda_{t\wedge T'_R}))}{R}.$$
Hence, let $R\to+\infty$ to have that for all $t>0$, $$\mP(T'\le t)=0.$$ 
So $T'=+\infty$ almost surely which concludes the proof. 
\end{proof}

We can reformulate Theorem $\ref{th:explo}$.
\begin{theorem}[Non explosion of the Dyson model]\label{th:explodyson}
For any initial data $\lambda_0\in \R^N_>$, let $(\lambda_t^i)_{1\le i\le N,\, t\ge 0}$ be the solution to $~\eqref{dysononcirc}$ which starts from $\lambda_0\in \R^N_>$. 
\newline
For all $1\le k\le N$, for all $t\ge 0$ we define $Z_t^k=\exp(i\lambda_t^k)$ and $T=\inf\{t\ge 0:  Z_t^i=Z_t^j \text{ for } 1\le i\ne j\le N\}$. Then $T=+\infty$ almost surely. 
\end{theorem} 

\begin{remark}
Let us briefly explain how to give a sense to a solution to \eqref{dysononcirc} when $\lambda_0$ is in $\R^N$ and not necessarily in $\R^N_>$. The following argument is used in Theorem 4.3.2 in \cite{anderson2010}. 
It is known that for all $t>0$, almost surely the spectrum of $D_N(t)$ is simple. This is a consequence of the fact that the spectrum of a matrix is simple if and only if the discriminant of the characteristic polynomial of this matrix is equal to 0 and the coefficients of the characteristic polynomial of a matrix are polynomials in the entries of this matrix. 
Let us recall that for all $t\ge 0$, we denoted $(\exp(i\lambda_t^j))_{1\le j\le N}$ the spectrum of $D_N(t)$. For all $t$ we ordered the eigenvalues such that $(\lambda_t^j)_{1\le j\le N}\in \overline{\R^N_>}$.
\newline
Let $(t_n)_{n\ge 0}\in (\R^+)^\N$ be a sequence that goes to 0 when $n$ goes to $\infty$. Almost surely for all $n$, $(\lambda_{t_n}^i)_{1\le i\le N}\in\R^N_>$. By Theorem \ref{th:explodyson}, for all $n$ we can consider the unique strong solution to \eqref{dysononcirc} starting from $(\lambda_{t_n}^i)_{1\le i\le N}$, $(\tilde\lambda_t)_{t\ge t_n}$ whose law coincides with the law of the eigenvalues of $(D_N(t))_{t\ge t_n}$ by the Dyson theorem \cite{dyson1962}. By uniqueness of the trajectory obtained in Theorem \ref{th:explodyson}, we can construct a solution $(\tilde\lambda_t)_{t>0}$ of the Dyson equation \eqref{dysononcirc} whose law coincides with the law of the eigenvalues of $(D_N(t))_{t>0}$ and which is in $\R^N_>$ for all $t>0$. Moreover the Hoffman-Wielandt inequality for normal matrices implies that $\lambda_t$ converges to $\lambda_0$ when $t$ goes to 0 by continuity in 0 of $(D_N(t))_{t\ge 0}$. This proves existence and uniqueness in law of solutions of \eqref{dysononcirc} starting from $\lambda_0$.  
\end{remark}

\section{Study of the Dyson partial differential equation on the circle}
Formally, thanks to Itô's formula, if $(\lambda_t^i)_{1\le i\le N}$ is a solution to ($\ref{dysononcirc}$), every limit point $m\in C([0,+\infty),\mesT)$ for the weak topology of the empirical measure of the system of particles: $$\mu_N(.)=\cfrac{1}{N}\sum_{i=1}^N\delta_{\lambda_.^i}\in C([0,+\infty),\mesT)$$ is solution to the following equation: 
\begin{equation}
\label{dyson}
 \partial_t m+\partial_\theta(mH[m])=0\,\,  \text{in } (0,\infty)\times \T, 
\end{equation}
where $\T=\R/2\pi\Z$, $\mesT$ is the set of probability measures on $\T$ and $H[m]$ is the Hilbert transform on the torus of the measure $m\in\mesT$. We shall detail more precisely in the next part what this equation means.
We call $\eqref{dyson}$ the Dyson equation on the circle.
\newline
The goal of this section is to study this equation and the primitive equation associated to it as in \cite{bertucci2022spectral}.
Then, in the next section we shall detail in which sense the system of particles converges to the solution to the Dyson equation on the circle. 

\subsection{Notation} 
A function $f:\R\to\R$ is an element of $C^k(\T)$ for $k\in[0,\infty]$ (resp. $L^p(\T)$ for $p\in[1,\infty]$) if $f$ is a $2\pi$ periodic function of class $C^k$ (resp. if $f$ is a $2\pi$ periodic function in $L^p([0,2\pi])$). We write $\mathcal F(\T)$ for the space of functions on the circle that means the functions from $\R$ to $\R$ that are $2\pi$ periodic.
\newline
Let $\mathcal F_{2\pi}$ (resp. $\mathcal F_{t,2\pi}$) be the space of functions $F$ defined on $\R$ (resp. $\R^+\times \R$) for which there exists $a\in\R$ such that for all $x\in\R$: $F(x+2\pi)=F(x)+a$ (resp. for all $t\ge 0$, for all $x\in\R$, we have $F(t,x+2\pi)=F(t,x)+a$). $\mathcal F_{2\pi}$ (resp. $\mathcal F_{t,2\pi}$) is a $\R$ vector space. 
\newline
We say that $F\in \mathcal F_{2\pi}$ (resp. $\mathcal F_{t,2\pi}$) is bounded if there exists exists $C>0$ such that $F$ is bounded by $C$ on $[0,2\pi]$ (resp. for all $t\ge 0$, $F(t,.)$ is bounded by $C$ on $[0,2\pi]$). 
\newline
Let $C_x^{1,1}(\R^+\times\R)$ be the space of functions $f(t,x)$ defined from $\R^+\times \R$ with values in $\R$ such that $f$ is a $C^1$ function and $\partial_x f$ is Lipschitz. 
\newline
Let $C^{1,1}_{2\pi}$ be the space of functions $F\in C^1(\R,\R)$ such that $F'$ is a Lipschitz function and for which there exists $a\in\R$ such that for all $x\in\R$, we have $F(x+2\pi)=F(x)+a$. $C^{1,1}_{2\pi}$ is a $\R$ vector space. 
\newline
Let $C^{1,1}_{t,2\pi}$ be the space of functions of $F\in C_x^{1,1}(\R^+\times\R)$ , for which there exists $a\in\R$ such that for all $t\ge 0$, for all $x\in\R$, we have $F(t,x+2\pi)=F(t,x)+a$. $C^{1,1}_{t,2\pi}$ is a $\R$ vector space.
\newline
\newline
\tab A priori, the Dyson equation \eqref{dyson} has to be understood in the sense of distributions. 
Let us recall that for an element $\mu\in\mesT$, we define $H[\mu]$ as the following distribution: 
$$H[\mu]=P.V.(\cotan(\theta/2))\ast \mu, $$ where $P.V.$ denotes the principal value.
More precisely we have that for all $\phi\in C^\infty(\T)$, $$\langle H[\mu],\phi\rangle=\int_\T\int_\T\phi(\theta')\cotan((\theta'-\theta)/2)\mu(d\theta)d\theta',$$ where integrals on $\T$ are understood in the sense of the principal value. 
When the measure $\mu$ has a smooth density with respect to Lebesgue measure we can explicitly write that: $$H[\mu](\theta)=\int_\T\cotan((\theta-\theta')/2)\mu(d\theta'):=\lim_{\varepsilon\to 0}\int_{\T\cap|\theta'-\theta|>\varepsilon}\cotan((\theta-\theta')/2)\mu(d\theta').$$
Note that the term $mH[m]$ in the Dyson equation does not make sense as the product of two distributions but rather as the distribution defined by: $$\langle m H[m],\phi\rangle=\int_\T\int_\T(\phi(\theta')-\phi(\theta))\cotan((\theta'-\theta)/2)m(d\theta)m(d\theta'), $$ for all $\phi\in C^\infty(\T)$.
For instance with this definition we have that $\delta_0H[\delta_0]=2\delta'_0$.
\newline
Hence, $\mu$ is a solution to $~\eqref{dyson}$ if for all $t\ge 0$, $m(t,.)\in\mesT$ and for all $\phi\in C^\infty(\T)$, it holds that for $t\ge0$: $$\langle m(t,.),\phi\rangle=\langle m(0,.),\phi\rangle+\int_0^t\langle m(s,.)H[m(s,.)],\phi'\rangle ds, $$ where $\langle \nu,f\rangle=\int_\T fd\nu$ for $\nu\in\mesT$ and $f\in L^1(\mu)$.  
We can obviously restrict ourselves to $\phi\in C^2(\T)$ in the previous formulation.
\newline
Let us recall an usual way to define the principal value as a real integral. By antisymmetric property of $\cotan(\theta)$ we have that $\int_\T \cotan(\theta/2)d\theta=0$. Hence for a suitable $f$ we have: 
\begin{equation}
H[f](\theta)=\lim_{\varepsilon\to 0}\int_{\T\cap|\theta'-\theta|>\varepsilon}f(\theta')\cotan((\theta-\theta')/2)d\theta'=\cfrac{1}{2}\int_\T\cotan(\theta'/2)(f(\theta-\theta')-f(\theta+\theta'))d\theta',
\label{hilb}
\end{equation} where the last integral is well defined if for instance $f\in C^2(\T)$.

\subsection{Primitive equation}
Following \cite{bertucci2022spectral}, we want to look for a primitive equation of \eqref{dyson}. What we mean by primitive equation, is an equation such that if $F$ is smooth and a solution to this primitive equation then $\partial_\theta F(t,\theta)$ shall be a solution to $\eqref{dyson}$.
\newline
However, there is no canonical way to consider a primitive of $\mu\in \mesT$, as the cumulative distribution is for a real measure, since there is no "good" order on $\T$. 
For $\mu\in \mesT$, let us define $F(\theta)=\mu([0,\theta])$ if $\theta\ge 0$ and $F(\theta)=-\mu((\theta,0) )$ if $\theta<0$. $F$ is non decreasing, right continuous and satisfies that $F(x+2\pi)=F(x)+1$ since $\mu$ is a probability measure on $\T$. 
Formally the primitive equation that we want to consider by integrating $~\eqref{dyson}$ is: 
\begin{equation}
\partial_t F+(\partial_\theta F) H[\partial_\theta F]=0
\label{primi1}
\end{equation}
We define the operator $A_0(F)=H(F')$ for $F\in C^{1,1}_{2\pi}$. By integration by part one can show that for a smooth $F$: $$A_0[F](\theta)=\int_{\T}\cfrac{2F(\theta)-F(\theta-\theta')-F(\theta+\theta')}{\sin^2(\theta'/2)}\cfrac{d\theta'}{4}=\int_{-\pi}^{\pi}\cfrac{F(\theta)-F(\theta-\theta')}{\sin^2(\theta'/2)}\cfrac{d\theta'}{2}.$$
\begin{remark}
The second expression above for $A_0$ is less practical because we cannot replace the integral on $[-\pi,\pi]$ by an integral on $\T$ since the integrated term is not $2\pi$ periodic. Moreover at first sight this integral seems to be ill-defined and it shall be understood as: $$A_0[F](\theta)=\int_{-\pi}^{\pi}\cfrac{F(\theta)-F(\theta-\theta')-\theta'F'(\theta)}{\sin^2(\theta'/2)}\cfrac{d\theta'}{2}.$$
\end{remark}
Hence, the primitive equation of $\eqref{dysononcirc}$ we shall consider is:
\begin{equation}
\partial_t F(t,\theta)+(\partial_\theta F(t,\theta)) A_0[F(t,.)](\theta)=0\,\, \text{in } (0,+\infty)\times \R,
\label{Primitive}
\end{equation}
where $A_0[F(t,.)]$ is:$$A_0[F(t,.)](\theta)=\int_{\T}\cfrac{2F(t,\theta)-F(t,\theta-\theta')-F(t,\theta+\theta')}{\sin^2(\theta'/2)}\cfrac{d\theta'}{4}=\int_{-\pi}^{\pi}\cfrac{F(t,\theta)-F(t,\theta-\theta')}{\sin^2(\theta'/2)}\cfrac{d\theta'}{2}.$$
The equation $\eqref{Primitive}$ shall be called the primitive Dyson equation.
\newline
The important property of the operator $A_0$ is that it satisfies a maximum principle: if $F\in C^{1,1}_{2\pi}$ has a maximum in $\theta_0$, then $A_0[F](\theta_0)\ge 0$ (by linearity of $A_0$ we also have that if $\phi\in C^{1,1}_{2\pi}$ and $F\in C^{1,1}_{2\pi}$ are such that $F-\phi$ has a maximum in $\theta _0$ then $A_0[F](\theta_0)\ge A_0[\phi](\theta_0)$) which motivates the approach by viscosity solution in \cite{bertucci2022spectral,bertucci2024spectral}. 
\begin{remark}
The operator $A_0$ appears in many problems as quasi-geostrophic equations, dislocation, porous media (see e.g. \cite{cordoba2004maximum,kiselev2022,do2018,de2025clogged,biler2008nonlinear}) and is linked with the fractional Laplacian \cite{daoud2022fractional}. Indeed, $F\in \mathcal C^2(\T)$ we can prove $(|c_n(A_0[F])|)_{n\in\Z}$ is proportional to $(|n||c_n(F)|)_{n\in\Z}$, where $c_n(F)$ is the $n^{th}$ Fourier coefficient of $F$. Thanks to this property the operator $A_0$ is defined as the square root of the Laplacian: $$A_0=\left(-\cfrac{d^2}{d\theta^2}\right)^{1/2}.$$ 
\end{remark}
Since we look at non decreasing in space variable solutions of $\eqref{Primitive}$ we can formally change the equation $\eqref{Primitive}$ in:  
\begin{equation}
\partial_t F(t,\theta)+(\partial_\theta F(t,\theta))_+ A_0[F(t,.)](\theta)=0\,\, \text{in } (0,+\infty)\times \R,
\label{eqPrimitive}
\end{equation}
This equation shall be treated with the theory of viscosity solutions \cite{crandall1992}.
\begin{definition}
\label{def viscosity sol}
$\bullet$ An upper semi continuous (usc) function $F\in \mathcal F_{t,2\pi}$ is said to be a viscosity subsolution
of $~\eqref{eqPrimitive}$ if for any function $\phi\in C^{1,1}_{t,2\pi}$, $(t_0,\theta_0)\in (0,\infty)\times \R$ point of local maximum of $F-\phi$ the following holds: 
\begin{equation}
\partial_t \phi(t_0,\theta_0)+(\partial_\theta \phi(t_0,\theta_0))_+ A_0[\phi(t_0,.)](\theta_0)\le 0
\label{subsol}
\end{equation}
\newline
$\bullet$ A lower semi continuous (lsc) function $F\in \mathcal F_{t,2\pi}$ is said to be a viscosity supersolution
of $~\eqref{eqPrimitive}$ if for any function $\phi\in C^{1,1}_{t,2\pi}$, $(t_0,\theta_0)\in (0,\infty)\times \R$ point of local minimum of $F-\phi$ the following holds: 
\begin{equation}
\partial_t \phi(t_0,\theta_0)+(\partial_\theta \phi(t_0,\theta_0))_+ A_0[\phi(t_0,.)](\theta_0)\ge 0
\label{supersol}
\end{equation}
\newline
$\bullet$ A viscosity solution $F$ of $~\eqref{eqPrimitive}$ is an usc subsolution such that $F_*$ is a supersolution where $F_*(t,\theta)=\lim\inf_{0\le s\to t, y\to \theta}F(s,y)$.
\newline
$\bullet$ A function $F$ is a viscosity solution to $\eqref{eqPrimitive}$ with initial data $F_0\in \mathcal F_{2\pi}$ if it is a viscosity solution of $~\eqref{eqPrimitive}$ and $F^*(0,.)\le F_0^*$ and ${(F_0)}_*\le F_*(0,.)$.
\newline
$\bullet$ A function $F$ such that for all $t\ge 0$, $F(t,.)$ is non decreasing, is a viscosity solution of $\eqref{Primitive}$ if it is a viscosity solution of $~\eqref{eqPrimitive}$.
\end{definition}
\begin{remark}
Since we are interested in non decreasing solutions $F$ in space variable, one can replace in the definition of sub and supersolution the positive part of $\partial_\theta \phi(t_0,\theta_0)$ by $\partial_\theta \phi(t_0,\theta_0)$ itself. Indeed, for any $t_0,h\ge 0$, if $\theta_0$ is a point of maximum of $F(t_0,.)-\phi(t_0,.)$ then: $$F(t_0,\theta_0)-\phi(t_0,\theta_0+h)\le F(t_0,\theta_0+h)-\phi(t_0,\theta_0+h)\le F(t_0,\theta_0)-\phi(t_0,\theta_0)$$ where we first used that $F(t_0,.)$ is non decreasing and then the fact that $\theta_0$ is a point of maximum of $F(t_0,.)-\phi(t_0,.)$. Hence, we deduce that $\partial_\theta \phi(t_0,\theta_0)\ge 0$.
\end{remark}
As usual in the theory of viscosity solution for integro -differential equations, it is more convenient to work with a more local reformulation. 
We introduce local and non local operator associated to $A_0$ as in \cite{arisawa2006,barles2008,bertucci2022spectral}. For $\delta>0$ we define the operators on $C^{1,1}_{2\pi}$, $I_{1,\delta}$ and $I_{2,\delta}$ with: 
\begin{align*}
I_{1,\delta}[\phi](\theta)&=\int_{\T\cap|\theta'|\le \delta}\cfrac{2\phi(\theta)-\phi(\theta-\theta')-\phi(\theta+\theta')}{\sin^2(\theta'/2)}\cfrac{d\theta'}{4}\\
I_{2,\delta}[\phi](\theta)&=\int_{\T\cap|\theta'|> \delta}\cfrac{2\phi(\theta)-\phi(\theta-\theta')-\phi(\theta+\theta')}{\sin^2(\theta'/2)}\cfrac{d\theta'}{4}.
\end{align*}

\begin{proposition}
Let $F$ be a subsolution (resp. supersolution) of $~\eqref{eqPrimitive}$. Then for all $\phi\in C^{1,1}_{t,2\pi}$, $\delta>0$ and $(t_0,\theta_0)\in (0,+\infty)\times \R$ such that $(F-\phi) (t_0,\theta_0)=0$ and $(F-\phi)(t,\theta)\le 0$ (resp. $\ge  0$) for any $(t,\theta)\in B((t_0,\theta_0),\delta)$, the following holds: $$\partial_t \phi(t_0,\theta_0)+(\partial_\theta \phi(t_0,\theta_0))_+\left( I_{1,\delta}[\phi(t_0,.)](\theta_0)+I_{2,\delta}[F(t_0,.)](\theta_0)\right)\le 0\text{ (resp. }\ge 0).$$
\end{proposition}

We can state a comparison principle for $\eqref{Primitive}$ as it is expected from viscosity theory with partial integro-differential term \cite{arisawa2006,barles2008}. We prove this functional comparison principle in our setting. 
\begin{theorem}
\label{thm:comparison principle}
Assume that $u\in \mathcal F_{t,2\pi}$ and $v\in \mathcal F_{t,2\pi}$ are respectively bounded viscosity subsolution and supersolution to $~\eqref{eqPrimitive}$. If $u(0,.)\le v(0,.)$, then for all time $t$, $u(t,.)\le v(t,.)$.
\end{theorem}
\begin{proof}
Let $\gamma>0$, instead of considering $u$ we could consider $u_{\gamma}(t,\theta)=u(t,\theta)-\gamma t$ which is a $\gamma$ strict subsolution in the sense that if $\phi\in C^{1,1}_{t,2\pi}$, $\delta>0$ and $(t_0,\theta_0)\in (0,+\infty)\times \R$ are such that $(u_{\gamma}-\phi) (t_0,\theta_0)=0$ and $(u_{\gamma}-\phi)(t,\theta)\le 0$ then for any $(t,x)\in B((t_0,x_0),\delta)$, the following holds: $$\partial_t \phi(t_0,\theta_0)+(\partial_\theta \phi(t_0,\theta_0))_+\left( I_{1,\delta}[\phi(t_0,.)](\theta_0)+I_{2,\delta}[u_{\gamma}(t_0,.)](\theta_0)\right)\le -\gamma <0.$$
Hence if we prove that for all $t,\theta$, $u_{\gamma}(t,\theta)\le v(t,\theta)$ for all $\gamma>0$ , then by taking the limit $\gamma\to 0$ we shall recover $u\le v$. Hence, we can assume $u$ is a $\gamma$ strict subsolution in the proof and not just a subsolution. 
\newline
We argue by contradiction and suppose that there exists $t_0>0$, $\theta_0\in \R $ such that $u(t_0,\theta_0)> v(t_0,\theta_0)$. Let $T>t_0$ and let use the classical technique of doubling variables. Thanks to the hypothesis, there exists $\alpha>0$ such that for all $\varepsilon>0$:
$$\sup\{u(t,\theta)-v(s,\theta')-\cfrac{1}{2\varepsilon}(\theta-\theta')^2-\cfrac{1}{2\varepsilon}(t-s)^2, \, t,s\in [0,T], \, \theta,\theta'\in \R^2\}>\alpha.$$
We shall justify that this supremum is actually a maximum by an usual localisation argument. Indeed, for a bounded $F\in \mathcal F_{t,2\pi}$ we have $a\in \R$ such that $F(t,\theta+2\pi)=F(t,\theta)+a$ for all $t,\theta$ and so we can deduce that there exists $K>0$ such that for all $t\ge 0$, $\theta\in \R$, $|F(t,\theta)|\le K(1+|\theta|)$. 
\newline
So for $u$ and $v$ in $\mathcal F_{t,2\pi}$ and bounded there exists $L>0$ such that for all $(t,s)\in \R^+$, for all $(\theta,\theta')\in \R^2$, $u(t,\theta)-v(s,\theta')\le L(1+|\theta|+|\theta'|)$ (this fact shall be called sublinearity). 
\newline
Now we can consider for $\beta>0$, $\varepsilon >0$ the following supremum: $$\sup\{u(t,\theta)-v(s,\theta')-\cfrac{1}{2\varepsilon}(\theta-\theta')^2-\cfrac{1}{2\varepsilon}(t-s)^2-\beta(|\theta|^2+|\theta'|^2), \, t,s\in [0,T], \, \theta,\theta'\in \R^2\}.$$
For $\beta>0$ small enough this supremum is greater than $\alpha/2>0.$ Moreover thanks to the sublinearity and the fact that $u$ is usc and $v$ is lsc, this supremum is a maximum.
\newline
Let $t^*,s^*,\theta^*,\theta'^*\in [0,T]^2\times \R^2$  be a point of maximum. We can classically assume that for $\varepsilon$ small enough, $t^*$ and $s^*$ are positive thanks to the fact that $\alpha>0$.
\newline
We want to use as tests functions: $$\phi_1(t,\theta)=v(s^*,\theta'^*)+\cfrac{1}{2\varepsilon}(\theta-\theta'^*)^2+\cfrac{1}{2\varepsilon}(t-s^*)^2+\beta(|\theta|^2+|\theta'^*|^2),$$ for $u$ and: $$\phi_2(s,\theta')=u(t^*,\theta^*)-\cfrac{1}{2\varepsilon}(\theta^*-\theta')^2-\cfrac{1}{2\varepsilon}(t^*-s)^2-\beta(|\theta^*|^2+|\theta'|^2),$$ for $v$.
\newline
The matter is of course that these functions are not in $C^{1,1}_{t,2\pi}$.
\newline
Take a small $\delta>0$ that will be specified later on.  We can find a function $\tilde \phi_1$ such that $\tilde \phi_1$ is equal to $\phi_1$ in $B((t^*,\theta^*),\delta)$ and which is in $C^{1,1}_{t,2\pi}$. Indeed for $\delta<\pi/2$, by localisation we can construct $\tilde \phi_1$ such that $\tilde \phi_1$ is equal to $\phi_1$ in $B((t^*,\theta^*),\delta)$ and impose that $\tilde\phi_1(t,\theta^*+\pi)=\tilde\phi_1(t,\theta^*-\pi)+1$ for all $t$ to define $\tilde\phi_1(t,.)$ on $[\theta^*-\pi,\theta^*+\pi]$ and then extend periodically $\tilde\phi_1(t,.)$ for all $t>0$.
We can do the same for $\phi_2$. We still call $\phi_1$ and $\phi_2$ these modifications of $\phi_1$ and $\phi_2$ that are in $C^{1,1}_{t,2\pi}$. 
\newline
Moreover to lighten the computations that shall follow we can forget the factor in $\beta$ since for instance if we look at $\gamma_\beta(\theta)=\beta(|\theta|^2+|\theta'^*|^2)$ we have that $\gamma_\beta''$ is uniformly bounded in $\theta$ by $2\beta$ and so we have that $I_{1,\delta}[\gamma_\beta](\theta'^*)$ converges uniformly in $\delta$ to 0 when $\beta$ converges to 0. So if we let $\beta$ converges to 0 first in the sub viscosity formulation this part disappears and same for the super viscosity formulation \cite{barles2008}. We still write $\phi_1$ and $\phi_2$ these functions without the term $\beta(|\theta|^2+|\theta'|^2)$.
\newline
By using $\phi_1$ and $\phi_2$ as test functions in the definition of subsolution and supersolution we have that: $$\cfrac{1}{\varepsilon}(t^*-s^*)+\cfrac{1}{\varepsilon}(\theta^*-\theta'^*)_+\left( I_{1,\delta}[\phi_1(t^*,.)](\theta^*)+I_{2,\delta}[u(t^*,.)](\theta^*)\right)\le -\gamma$$
and: $$\cfrac{1}{\varepsilon}(t^*-s^*)+\cfrac{1}{\varepsilon}(\theta^*-\theta'^*)_+\left( I_{1,\delta}[\phi_2(s^*,.)](\theta'^*)+I_{2,\delta}[v(s^*,.)](\theta'^*)\right)\ge 0.$$
We subtract the first inequality to the second and we have: \begin{equation}
\cfrac{1}{\varepsilon}(\theta^*-\theta'^*)_+\left( I_{1,\delta}[\phi_2(s^*,.)](\theta'^*)-I_{1,\delta}[\phi_1(t^*,.)](\theta^*)+I_{2,\delta}[v(s^*,.)](\theta'^*)-I_{2,\delta}[u(t^*,.)](\theta^*)\right)\ge \gamma.
\label{ineqfinal}
\end{equation}
Firstly, we look at the $I_{2,\delta}$ part: 
\begin{align*}
&I_{2,\delta}[v(s^*,.)](\theta'^*)-I_{2,\delta}[u(t^*,.)](\theta^*)=\\
&\int_{\T\cap|z|>\delta}\cfrac{v(s^*,\theta'^*)-v(s^*,\theta'^*+z)-v(s^*,\theta'^*-z)-(2u(t^*,\theta^*)-u(t^*,\theta^*+z)-u(t^*,\theta^*-z))}{\sin^2(z/2)}\cfrac{dz}{4}.
\end{align*}
But using that $(t^*,s^*,\theta^*,\theta'^*)$ is a point of maximum we have that for all $z\in \R$: 
\begin{align*}
u(t^*,\theta^*)-v(s^*,\theta'^*)\ge u(t^*,\theta^*+z)-v(s^*,\theta'^*+z)\\
u(t^*,\theta^*)-v(s^*,\theta'^*)\ge u(t^*,\theta^*-z)-v(s^*,\theta'^*-z).
\end{align*}
By adding these two inequalities we see that the previous numerator is actually non negative. Hence we have: $$I_{2,\delta}[v(s^*,.)](\theta'^*)-I_{2,\delta}[u(t^*,.)](\theta^*)\le 0.$$
Then for the terms in $I_{1,\delta}$ we notice that since $\phi_1(t^*,.)$ (resp. $\phi_2(s^*,.)$) is bounded in $C^2$ by $\varepsilon^{-1}$ in $B(\theta'^*,\delta)$ (resp. $B(\theta^*,\delta)$), we have: $$I_{1,\delta}[\phi_2(s^*,.)](\theta'^*)-I_{1,\delta}[\phi_1(t^*,.)](\theta^*)\le  \cfrac{2\delta}{\varepsilon}. $$
Hence, using these two results in $~\eqref{ineqfinal}$ we have that for $\delta,\varepsilon$ small enough: $$\cfrac{2\delta}{\varepsilon}\cfrac{(\theta^*-\theta'^*)_+}{\varepsilon}\ge \gamma.$$
We chose $\delta=\varepsilon^{3/2}$ and since we know from classical viscosity solution \cite{crandall1992} that $(\theta^*-\theta'^*)^2/\varepsilon \longrightarrow_{\varepsilon \to 0} 0$, we deduce that: $$\cfrac{2\delta}{\varepsilon^{3/2}}\cfrac{(\theta^*-\theta'^*)_+}{\sqrt{\varepsilon}}\longrightarrow_{\varepsilon \to 0} 0.$$
We obtain that $0 \ge \gamma$ which is a contradiction. 
\end{proof}

\subsection{Uniqueness of viscosity solution}
Thanks to comparison principle we can classically obtain uniqueness of viscosity solution for $~\eqref{Primitive}$. The proofs are similar with the real Dyson case \cite{bertucci2022spectral,bertucci2024spectral}.
\begin{corollary}
Let $F$ be a bounded viscosity solution to $~\eqref{eqPrimitive}$ with initial condition $F_0\in\mathcal F_{2\pi}$. We suppose that $F_0$ is non decreasing, usc and bounded. Then $F(t,.)$ is non decreasing for all time. 
\end{corollary}
\begin{proof}
By monotonicity of $F_0$, since the initial condition of $F$ is $F_0$ we get that for all $\varepsilon>0,\, \varepsilon'>0$, for all $x\in\R$: $$F(0,x)\le F_0(x)\le (F_0(x+\varepsilon'))_*+\varepsilon\le F_*(0,x+\varepsilon')+\varepsilon .$$
Using Theorem \ref{thm:comparison principle} with $u=F$ as a subsolution and $v(t,x)=F_*(t,x+\varepsilon')+\varepsilon$ as a supersolution we get that for all $\varepsilon>0,\, \varepsilon'>0$, $\forall t\ge 0$, $\forall x\in\R$:
\begin{equation}
\label{eq: ineg croissance F}
F(t,x)\le F_*(t,x+\varepsilon')+\varepsilon\le F(t,x+\varepsilon')+\varepsilon.
\end{equation}
For $y>x$, we choose $\varepsilon'=y-x$ in \eqref{eq: ineg croissance F}. Let $\varepsilon$ goes to $0$ to obtain that for all $t\ge 0$, for all $x<y$, $F(t,x)\le F(t,y)$.
\end{proof}
\begin{theorem}
Given a non-decreasing usc and bounded function $F_0\in \mathcal F_{2\pi}$, there exists at most one bounded viscosity solution $F$ of $~\eqref{Primitive}$ with initial condition $F_0$.
\end{theorem}
\begin{proof}
Assume that there exists $F$ and $G$ which are viscosity solutions of $\eqref{eqPrimitive}$ with initial condition $F_0$.
Let $\varepsilon>0$. Since $F$ and $G$ have the same initial condition $F_0$ we get that for all $x\in\R$: 
$$F(0,x)\le F_0(x)\le (F_0(x+\varepsilon))_*+\varepsilon\le G_*(0,x+\varepsilon)+\varepsilon.$$
Using Theorem \ref{thm:comparison principle} with $u=F$ as a viscosity subsolution and $v_\varepsilon(t,x)=G_*(t,x+\varepsilon)+\varepsilon$ as a supersolution, we get that for all $t\ge0$, for all $x\in\R$: $$F(t,x)\le G_*(t,x+\varepsilon)+\varepsilon.$$
Letting $\varepsilon$ goes to $0$ yields that for all $t\ge 0$, for all $x\in\R$, $$F(t,x)\le G(t,x).$$
By symmetry of $F$ and $G$ we get that $F=G$.
\end{proof}
\begin{remark}
Let us notice that in the previous comparison principle we do not have an hypothesis like Lipschitz for one of the two functions we compare. Hence, we do not have to prove existence of a solution to $\eqref{Primitive}$ starting from a smooth initial data as it was done in \cite{bertucci2022spectral}. 
\end{remark}

\section{Convergence of the system of particles}

\subsection{Periodic system of $N$ particles}
In Section \ref{se:matrix to part} we proved the existence of the system of $N$ particles of the unitary Dyson case. To state a discrete comparison principle, we shall introduce a periodic system of $N$ particles obtained by periodizing the system of $N$ particles. Indeed, stating a comparison principle for two systems of particles is not clear due to the fact that there is no global order on $\T$. We will define a periodization of the system to state that we will be able to compare two systems of particles on $\T$, solutions of \eqref{dysoncircleeq} for a suitable angle of reference.
\newline
From now on, for an integer $i\in \Z$ we write $i[N]$ the unique integer in $[0,N)$ such that $N$ divides $i-i[N]$.
\begin{definition}
\label{def:firstdysonperio}
A system of real particles $(\lambda_.^i)_{i\in \Z}$ satisfies the $N$ periodic Dyson system if it satisfies the following system:
\begin{equation}
\left\{ 
\begin{split}
\forall i\in \Z,\,\, &d\lambda_t^i=\cfrac{1}{N}\dis\sum_{k=1}^{N-1}\cotan((\lambda_t^i-\lambda_t^{i+k})/2)\,dt+\cfrac{2}{\sqrt{N}}\,dB_t^{i[N]},\\
\forall i\in \Z,\,\forall t\ge 0,\, &\lambda_t^i=\lambda_t^{i[N]}+2\left\lfloor \cfrac{i}{N}\right\rfloor\pi,
\end{split}
\label{dysoncircper}
\right.
\end{equation}
with $(B_.^i)_{0\le i\le N-1}$ a family of independent Brownian motions.
\end{definition}
Let us remark that thanks to the fact that we impose to our particles to have a $2\pi$ periodic dynamic we can  write: 
\begin{align*}
\forall i\in \Z,\,\, d\lambda_t^i&=\cfrac{1}{N}\dis\sum_{k=1}^{N-1}\cotan((\lambda_t^i-\lambda_t^{i+k})/2)\,dt+\cfrac{2}{\sqrt{N}}\,dB_t^{i[N]}\\
&=\cfrac{1}{N}\dis\sum_{\overline{j}\ne \overline{i}\in \Z/N\Z}\cotan((\lambda_t^{\overline{i}}-\lambda_t^{\overline{j}})/2)\,dt+\cfrac{2}{\sqrt{N}}\,dB_t^{i[N]}.
\end{align*}
For this system we have existence and uniqueness of a solution given a $N$ periodic initial data thanks to the existence and uniqueness of the unitary Dyson case of Theorem $\ref{th:explo}$.
\begin{proposition}
\label{prop:firstdef}
Let $0\le\lambda_0^0<\lambda_0^1<...<\lambda_0^{N-1}<2\pi$ be $N$ initial particles. We define for $i\in \Z$ the $i^{th}$ initial particle $\lambda_0^i=\lambda_0^{i[N]}+2\left\lfloor i/N\right\rfloor\pi$. 
\newline
There exists a unique solution to $\eqref{dysoncircper}$, $(\lambda_.^i)_{i\in \Z}$, defined for all time $t\ge 0$ and which starts from $(\lambda_0^i)_{i\in\Z}$.
\newline
Moreover, for all $k\in \Z$, for all $t\ge 0$ we define $Z_t^k=\exp(i\lambda_t^k)$ and $T=\inf\{t\ge 0 :  Z_t^i=Z_t^j \text{ for } i\ne j\}$. It satisfies $T=+\infty$ almost surely. 
\end{proposition}
\begin{proof}
Thanks to existence and uniqueness of the system of $N$ particles studied in Section $\ref{se:matrix to part}$, we know that we can define a solution to $\eqref{dysononcirc}$ starting from $(\lambda_0^0,...,\lambda_0^{N-1})\in\R^N_>$. This solution is defined for every time and is in $\R^N_>$ for every time. We define for every time $t\ge 0$ and for every $i\in \Z$, $\lambda_t^i=\lambda_t^{i[N]}+2\left\lfloor i/N\right\rfloor\pi$. We have that $(\lambda_.^i)_{i\in\Z}$ is solution to $\eqref{dysoncircper}$. 
\newline
The uniqueness is also clear since by the previous part the dynamic of $(\lambda_.^i)_{0\le i\le N-1}$ which starts from $(\lambda_0^0,...,\lambda_0^{N-1})\in\R^N_>$ is unique and since for every time $t\ge 0$ for every $i\in \Z$, $\lambda_t^i=\lambda_t^{i[N]}+2\left\lfloor i/N\right\rfloor\pi$ we have a unique solution to $\eqref{dysoncircper}$.
\newline
Finally about the last point we notice that for all $k\in \Z$, for all $t\ge 0$ $Z_t^k=\exp(i\lambda_t^k)=\exp(i\lambda_t^{k[N]})$. Since the trajectories of the particles are continuous we have that $T=\inf\{t\ge 0 :  Z_t^i=Z_t^j \text{ for } 0\le i\ne j\le N-1\}$ and so by Theorem $\ref{th:explodyson}$ we have $T=+\infty$ almost surely.
\end{proof}

We give another definition for a system of $N$ particles without imposing the $2\pi$ periodicity in the previous definition but to obtain it as a consequence of the study of the $N$ particles Dyson case. This new definition shall be consistent with Definition \ref{def:firstdysonperio}.
\newline
\newline
For $i\in \Z$ we write $k_{N,i}\in\Z$ the unique integer such that $i\in [k_{N,i} N, (k_{N,i}+1)N-1)$ and we denote by $\Z_{N,i}=\Z\cap[k_{N,i} N, (k_{N,i}+1)N-1)$. 
\begin{definition} 
A system of particles $(\lambda_.^i)_{i\in \Z}$ satisfies the $N$ periodic Dyson system if it satisfies the following system of SDE:
\begin{equation} 
\forall i\in \Z,\,\, d\lambda_t^i=\cfrac{1}{N}\dis\sum_{k\in \Z_{N,i}, \, k\ne i }\cotan((\lambda_t^i-\lambda_t^{k})/2)\,dt+\cfrac{2}{\sqrt{N}}\,dB_t^{i[N]},
\label{dysoncircper2}
\end{equation}
with $(B_.^i)_{0\le i\le N-1}$ a family of independent Brownian motions.
\end{definition}
\begin{remark}
We insist upon the fact that since a priori the $\lambda_t^i$ does not satisfy the periodic equality modulo $2\pi$ we can not write as before the sum on any $N-1$ class of disjoints particles.
\end{remark}
Thanks to existence and uniqueness of the $N$ particles Dyson case of Theorem $\ref{th:explo}$, there exists a unique solution to the previous system given a periodic initial data. Moreover this solution is the same as the unique solution to the first definition. 
\begin{proposition}
\label{pro:equivalence}
Let $0\le\lambda_0^0<\lambda_0^1<...<\lambda_0^{N-1}<2\pi$ be $N$ initial particles. We define for $i\in \Z$ the $i^{th}$ initial particle $\lambda_0^i=\lambda_0^{i[N]}+2\left\lfloor i/N\right\rfloor\pi$. 
\newline
Then there exists a unique solution to $\eqref{dysoncircper2}$, $(\lambda_.^i)_{i\in \Z}$, defined for all time $t\ge 0$ which starts from $(\lambda_0^i)_{i\in\Z}$. 
\newline
Furthermore almost surely this solution satisfies: $$\forall i\in \Z,\,\forall t\ge 0,\, \lambda_t^i=\lambda_t^{i[N]}+2\left\lfloor \cfrac{i}{N}\right\rfloor\pi.$$
\newline
Hence, the unique solution to $\eqref{dysoncircper}$ starting from $(\lambda_0^i)_{i\in\Z}$ is the same unique solution to $\eqref{dysoncircper2}$ starting from  $(\lambda_0^i)_{i\in\Z}$ almost surely.
\end{proposition} 
\begin{proof}
For the existence, we can look individually at the packs of $N$ particles $\{0,...,N-1\}$, $\{N,...2N-1\}$, $\{-N,...-1\}$ ... since we chose that the interactions of the $i$ particles only depend on the particles of $\Z_{N,i}$. Hence, using the existence part of the system of $N$ particles we can individually justify the existence for all time of these systems of $N$ particles. 
The problem is that maybe a particle from a pack collides with another particle. However that is not the case. Indeed we have that almost surely: $$\forall i\in \Z,\,\forall t\ge 0,\, \lambda_t^i=\lambda_t^{i[N]}+2\left\lfloor \cfrac{i}{N}\right\rfloor\pi.$$
For instance, if we look at the pack $(\lambda_.^i)_{i\in \{N,...,2N-1\}}$ we notice that $(\lambda_.^i-2\pi)_{i\in \{N,...,2N-1\}}$ is of course solution to the system of $N$ particles but now starting from the exact same initial data as $(\lambda_.^i)_{i\in \{0,...,N-1\}}$. Hence by uniqueness of the trajectories of the system of $N$ particles previously studied, we deduce that almost surely for every $i\in\{0,...,N-1\}$, for every time $t\ge 0$ we have: $\lambda_t^{i+N}=\lambda_t^i+2\pi$. We can do this for every packs of $N$ particles and since there is only a countable number of packs of particles, we conclude that almost surely $$\forall i\in \Z,\,\forall t\ge 0,\, \lambda_t^i=\lambda_t^{i[N]}+2\left\lfloor \cfrac{i}{N}\right\rfloor\pi.$$
Now since we know that the trajectories of every packs of $N$ particles are in $\R^N_>$ for every time by the previous study we have that for every time $t\ge 0$: $$...<\lambda_t^{-1}=\lambda_t^{N-1}-2\pi<\lambda_t^0<\lambda_t^1<...<\lambda_t^{N-1}<\lambda_t^0+2\pi=\lambda_t^{N}<\lambda_t^{N+1}<...$$ and so the system of infinite particles exists for every time and there is no collision. 
\newline
For the uniqueness of a solution to $\eqref{dysoncircper2}$ starting from $(\lambda_0^i)_{i\in\Z}$ we can notice, by exactly the same argument as in the existence part, that a solution $(\lambda_.^i)_{i\in\Z}$ of $\eqref{dysoncircper2}$ which starts from an initial condition $(\lambda_0^i)_{i\in\Z}$ which satisfies an hypothesis of periodic equality modulo $2\pi$ as in the statement, shall satisfy that almost surely: $$\forall i\in \Z,\,\forall t\ge 0,\, \lambda_t^i=\lambda_t^{i[N]}+2\left\lfloor \cfrac{i}{N}\right\rfloor\pi.$$
Hence the uniqueness is clear since the uniqueness of the system of the $N$ particles $\{0,...,N-1\}$ is clear by the uniqueness of the trajectory of a solution to the system of $N$ particles previously studied.
\newline
Finally the fact that the unique solution to $\eqref{dysoncircper2}$ starting from $(\lambda_0^i)_{i\in\Z}$ satisfies that almost surely: $$\forall i\in \Z,\,\forall t\ge 0,\, \lambda_t^i=\lambda_t^{i[N]}+2\left\lfloor \cfrac{i}{N}\right\rfloor\pi,$$ proves that is also the unique solution to $\eqref{dysoncircper}$ which starts from $(\lambda_0^i)_{i\in\Z}$. The reverse is obvious.
\end{proof}

\subsection{Discrete comparison principle} We can state a discrete comparison principle for the $N$ periodic Dyson equation that shall have a key impact in the proof of the convergence of the system of particles as in \cite{bertucci2022spectral}.

\begin{theorem}
\label{thm:comparison}
Let $a$ be a real number. 
Let $a\le\lambda_0^0<\lambda_0^1<...<\lambda_0^{N-1}<a+2\pi$ be $N$ initial particles. We define for $i\in \Z$ the $i^{th}$ initial particle $\lambda_0^i=\lambda_0^{i[N]}+2\left\lfloor i/N\right\rfloor\pi$. 
\newline
Let $a\le\mu_0^0<\mu_0^1<...<\mu_0^{N-1}<a+2\pi$ be $N$ initial particles. We define for $i\in \Z$ the $i^{th}$ initial particle $\mu_0^i=\mu_0^{i[N]}+2\left\lfloor i/N\right\rfloor\pi$. 
\newline
Let $(\lambda_.^i)_{i\in\Z}$ (resp. $(\mu_.^i)_{i\in\Z}$) be the unique strong solution to $\eqref{dysoncircper}$ (or equivalently $\eqref{dysoncircper2})$ which starts from $(\lambda_0^i)_{i\in\Z}$ (resp. $(\mu_0^i)_{i\in\Z}$) with the same Brownian motions $(B_.^i)_{0\le i\le N-1}$.
\newline
Assume that for all $i\in \Z$ we have:  $\lambda_0^i\le \mu_0^i$. 
\newline
Then for all time $t\ge 0$, for all $i\in \Z$, we have: $\lambda_t^i\le \mu_t^i$.
\end{theorem}
The proof is similar to the proof in the real Dyson case \cite{anderson2010,bertucci2022spectral,sniady2002random}.

\begin{proof}
We know that almost surely the trajectories of the particles $(\lambda_.^i)_{i\in\Z}$ and $(\mu_.^i)_{i\in\Z}$ does not collide and so are well defined for every time. 
\newline
Let $\varepsilon >0$ and let us define the stopping time $$\tau_\varepsilon=\inf\{t,\exists i\in\Z, \lambda_t^i-\mu_t^i>\varepsilon\}.$$
We shall prove that for all $T>0$ we have $\mP(\tau_\varepsilon<T)=0$.
\newline
We define for every $t\ge 0$ and every $i\in\Z$ $$w_t^i=\lambda_t^i-\mu_t^i-\delta t$$ with $\delta=\frac{\varepsilon}{2T}$ (remark that by the periodic equality modulo $2\pi$ of the two systems of particles, there are actually only $N$ $w_t^i$). By using the fact that $(\lambda_.^i)_{i\in\Z}$ and $(\mu_.^i)_{i\in\Z}$ are solutions of $\eqref{dysoncircper}$, we have that: $$dw_t^i=\cfrac{1}{N}\sum_{k=1}^{N-1}\left(\cotan((\lambda_t^i-\lambda_t^{i+k})/2)-\cotan((\mu_t^i-\mu_t^{i+k})/2) \right)dt-\delta.$$
Let $\omega\in \Omega$ such that $\tau_\varepsilon(\omega)<T$, we can find an index $i_0\in\Z$ (which of course depends on $\omega$) and a times $\tau>0$ (which also depends on $\omega$ too) such that: 
\begin{equation}
\left\{ 
\begin{split}
&\forall i\in\Z, \,\forall t\in[0,\tau]\,\,w_t^i\le 0\\
&\exists i_0,\, w_\tau^{i_0}=0\\
&dw_\tau^{i_0}\ge 0,
\end{split}
\right.
\end{equation}
since for all $i\in \Z$ we have:  $\lambda_0^i\le \mu_0^i$ and the trajectories of $w^i$ are continuous.
\newline
If we evaluate the previous formula for $dw_t^i$ at $t=\tau$ and $i=i_0$, we get: 
$$dw_\tau^{i_0}=\cfrac{1}{N}\sum_{k=1}^{N-1} \left[\cotan\left(\cfrac{\lambda_t^{i_0}-\lambda_t^{i_0+k}}{2}\right)-\cotan\left(\cfrac{\mu_t^{i_0}-\mu_t^{i_0+k}}{2}\right)\right]dt-\delta.$$
Now for all $1\le k\le N-1$, let $a_k=\lambda_\tau^{i_0}-\lambda_\tau^{i_0+k}$ and $b_k=\mu_\tau^{i_0}-\mu_\tau^{i_0+k}$.
We remark that by definition for all, for all $1\le k\le N-1$ we have $-2\pi< a_k< 0$, $-2\pi< b_k< 0$ and $b_k-a_k=-\delta\tau-\mu_\tau^{i_0+k}+\lambda_\tau^{i_0+k}=w_\tau^{i_0+k}\le 0$. 
Since $\cotan(./2)$ is non increasing on $(-2\pi,0)$ we deduce that for all $1\le k\le N-1$ $\cotan(a_k/2)\le\cotan(b_k/2)$ and so we have that for all $1\le k\le N-1$: $$\cotan\left(\cfrac{\lambda_t^{i_0}-\lambda_t^{i_0+k}}{2}\right)-\cotan\left(\cfrac{\mu_t^{i_0}-\mu_t^{i_0+k}}{2}\right)\le 0 $$
So we have: $dw_\tau^{i_0}\le-\delta<0$ which is a contradiction.
\newline
Hence, we get that for all $\varepsilon>0$ and for all $T>0$, almost surely for all $t<T$, for all $i\in \Z$, $\lambda_t^i\le\mu_t^i+\varepsilon$. 
\newline
By using a countable sequences of $\varepsilon>0$ and $T>0$ we deduce that almost surely for all $\varepsilon>0$ for all $t\ge 0$, for all $i\in \Z$, $\lambda_t^i\le\mu_t^i+\varepsilon$.
\newline
Hence, by letting $\varepsilon\to 0$ we conclude that almost surely for all $t\ge 0$, for all $i\in \Z$, $\lambda_t^i\le\mu_t^i$.
\end{proof}
\begin{remark}
This result is also true for periodized solutions of \eqref{dysoncircleeq}. This comparison principle can be used to prove existence and uniqueness of solution to \eqref{dysoncircleeq} with the hypothesis $\beta_N\ge N^2/2$ starting from an initial data in $\R^N$ as in Proposition 4.3.5 in \cite{anderson2010}  for the real case.
\end{remark}

\subsection{Convergence of the system of particles}
Thanks to the discrete and the functional comparison principle, we can have a similar approach as in the real Dyson case \cite{bertucci2022spectral}.
\newline
In this Section we call $(\Omega, \mathcal F,\mP)$ the probability space on which all the random variables are defined. 
\newline
\tab Let us recall that for $\mu\in \mesT$, we let $F_\mu(\theta)=\mu([0,\theta])$ if $\theta\ge 0$ and $F_\mu(\theta)=-\mu((\theta,0) )$ if $\theta<0$. Formally $F_\mu(\theta)=\int_0^\theta \mu(d\theta')$ if for instance $\mu$ has a density with respect to Lebesgue measure. $F$ is non decreasing, right continuous and satisfies that $F_\mu(x+2\pi)=F_\mu(x)+1$ since $\mu$ is a probability measure on $\T$. 
\newline
For $(\lambda_1,...,\lambda_N)\in\T^N$ we introduce the empirical measure $\mu_N\in\mesT$ associated to $(\lambda_1,...,\lambda_N)$ by:  $$\mu_N=\cfrac{1}{N}\sum_{i=1}^N\delta_{\lambda_i}.$$
This definition shall be also used if we consider a $N$ periodic system of particles $(\lambda_i)_{i\in \Z}\in\T^{\Z}$ (i.e. for all $i\in\Z$, $\lambda_i=\lambda_{i[N]}$ modulo $2\pi$) by defining: $$\mu_N=\cfrac{1}{N}\sum_{i\in\Z/N\Z}\delta_{\lambda_i}.$$ 
\newline 
For $a\in\R$, 
let $a\le\lambda_0^0<\lambda_0^1<...<\lambda_0^{N-1}<a+2\pi$ be $N$ initial particles. We define for $i\in \Z$ the $i^{th}$ initial particle $\lambda_0^i=\lambda_0^{i[N]}+2\left\lfloor i/N\right\rfloor\pi$. We consider $(\lambda_t^{i})_{i\in\Z}$ the unique solution to $\eqref{dysoncircper}$ which starts from $(\lambda_0^i)_{i\in \Z}$. We define for all $t\ge 0$, for all $\theta\in\R$, $$F_N(t,\theta)=F_{\mu_N(t)}(\theta),$$ where $\mu_N(t)$ is the empirical measure associated to $(\lambda_t^i)_{i\in \Z}$.
\newline
We also define the upper semi continuous function: $$F^*(t,\theta)=\underset{N\to\infty,\,t_N\to t^-,\, \theta_N\to\theta}{\limsup}F_N(t_N,\theta_N).$$ Let us remark that $F^*$ is a random variable since the $F_N$ are.
\newline
Now we can state the main convergence result.
\begin{theorem}[Convergence of the system of particles]\label{th:conv}
Assume that the empirical measure $\mu_N^0$ of initial conditions defined by: $$\mu_N^0=\cfrac{1}{N}\sum_{i=0}^{N-1}\delta_{\lambda_0^i}$$ converges almost surely for the weak convergence toward a measure $\mu_0\in\mesT$.
\newline
Then almost surely, $F^*$ is the unique viscosity solution of $\eqref{Primitive}$ which satisfies $F(0,\theta)=F_{\mu_0}(\theta)$ almost everywhere.
\end{theorem}
As explained in introduction the proof uses the same approach as for the real Dyson case in \cite{bertucci2022spectral}. The main difference is that we have to use the local formulation of viscosity solutions because locally we can see an arc of the circle as an interval of $\R$. Thanks to that formulation we shall construct system of particles associated to test functions that are "below" the initial system of particles and then we shall use the discrete comparison principle. Moreover, we need to clarify some details of the proof of Theorem 3 of \cite{bertucci2022spectral} which were missing.

\begin{proof}
We follow the idea of \cite{bertucci2022spectral}. The main idea of the proof is given a test function $\phi$ such that $F^*-\phi$ has a maximum in a point as in the viscosity subsolution formulation in Definition \ref{def viscosity sol}, for all $N>0$ we shall construct a periodic system of $N$ particles associate to $\phi$. Formally, since $\phi$ is above $F^*$ this periodic system of $N$ particles will be ordered in a way to apply the discrete comparison principle with the original system of particles $(\lambda_t^i)_{1\le i\le N}$. This shall give an information about the evolution of the test function $\phi$ along this flow of particles. On the other hand, we can express the evolution of $\phi$ along this flow thanks to the Itô formula using the explicit dynamic of the particles. Taking the limit when $N$ goes to  infinity shall give the relation $\eqref{subsol}$.
\newline
Let $(s_N)_{N\in\N}$ be the set of positive rational numbers.
\subsubsection{Construction of adapted systems of particles associated to a test function and Itô's formula}
\label{subsubsection construction particles}
Fix $N\in\N$, $\delta>0$, $t_0\in\R^+$, $\theta_0\in\R$ and $\phi\in C^{1,1}_{t,2\pi}$. 
We consider $I_\delta:=(\theta_0-\delta,\theta_0+\delta)$ and an interval $I$ such that $I_\delta\subset \overset{\circ} I$. We suppose that $\phi(s_N,.)$ is strictly increasing on $I_\delta$ and we let $J_\delta(t):=\phi(t,I_{\delta})$
We change the numbering of the $\lambda_{s_N}$ with respect to the interval $I$ which means that if $I$ is $[a,a+2\pi)$ for an $a\in \R$, we write: $$a\le \lambda_{s_N}^0< \lambda_{s_N}^1<...<\lambda_{s_N}^{N-1}<a+2\pi,$$ the inequalities are strict since the particles does not collide for any $t>0$ almost surely.
For all $i$, $\lambda_{s_N}^i $ mod $2\pi$ is associated to a unique class $i/N$ mod $1$. 
\newline
For the $i$ mod $N$ such that $i/N$ mod 1 is in $J_\delta(s_N)$ mod 1, we consider $\gamma_{s_N}^i$ mod $2\pi$ which is equal to $(\phi(s_N))^{-1}(i/N)$ mod $2\pi$ (this definition makes sense since $\phi(s_N,\theta+2\pi)=\phi(s_N,\theta)+1$ for all $\theta$ and $\phi(s_N,.)$ is strictly increasing on $I_\delta$ by hypothesis.)
By abuse of notations we write $\gamma_{s_N}^i$ the unique real in $I$ which represents the previous class modulo $2\pi$. More exactly $\gamma_{s_N}^i$ is in $I_\delta$ by definition of $J_\delta$. 
\newline
We fix $\sigma_N\in \mathfrak{S}_N$ where $\mathfrak{S}_N$ is the group of permutations of cardinal $N!$. 
We define a new system of particles associated to $\sigma_N$. $(\mu^{i,\sigma_N}_.)_{0\le i\le N-1}$ by the initial data given by: 
\begin{equation}
\left\{
\begin{split}
\mu_{s_N}^{i,\sigma_N}&=\lambda_{s_N}^i \text{ if } i/N \text{ mod 1 isn't in } J_\delta(s_N) \text { mod 1}\\
\mu_{s_N}^{i,\sigma_N}&=\min(\gamma_{s_N}^i,\lambda_{s_N}^i) \text{  if } i/N \text{ mod 1 is in } J_\delta(s_N) \text { mod 1}
\end{split}
\right.
\end{equation}
As before, we define for all $i\in \Z$, $$\mu_{s_N}^{i,\sigma_N}=\mu_{s_N}^{i[N],\sigma_N}+2\left\lfloor \cfrac{i}{N}\right\rfloor\pi.$$
Then let us define $(\mu^{i,\sigma_N}_{t})_{i\in\Z,\, t\ge s_N}$ as the unique strong solution to: 
\begin{equation}
\left\{
\begin{split}
\forall i\in \Z,\,\, &d\mu_t^{i,\sigma_N}=\cfrac{1}{N}\dis\sum_{k=1}^{N-1}\cotan((\mu_t^{i,\sigma_N}-\mu_t^{i+k,\sigma_N})/2)\,dt+\cfrac{2}{\sqrt{N}}\,dB_t^{\sigma_N(i[N])},\\
\forall i\in \Z,\,\forall t\ge s_N,\, &\mu_t^{i,\sigma_N}=\mu_t^{i[N],\sigma_N}+2\left\lfloor \cfrac{i}{N}\right\rfloor\pi,
\end{split}
\right.
\end{equation}
which starts from $(\mu_{s_N}^{i,\sigma_N})_{i\in\Z}$, where the $N$ Brownian motions are the same Brownian motions for the evolution of the system of $N$ particles of the $(\lambda_.^i)_{i\in\Z}$.
Applying the Itô formula gives that almost surely, for all $\phi\in C^{1,1}_{t,2\pi}$, for all $N\in\N$, for all $\sigma_N\in \mathfrak{S}_N$, for all $i\in\Z$, for all $t_0\ge s_N$: 
\begin{equation}
\label{eq: ito formula pas bien}
\begin{split}
\phi(t_0,\mu_{t_0}^{i})-\phi(s_N,\mu_{s_N}^{i})&=\int_{s_N}^{t_0}\partial_\theta\phi(t',\mu_{t'}^{i})\cfrac{1}{N}\dis\sum_{k=1}^{N-1}\cotan((\mu_{t'}^{i}-\mu_{t'}^{i+k})/2)dt'
\\
&+\int_{s_N}^{t_0}\partial_t\phi(t',\mu_{t'}^{i})dt'\\
&+\int_{s_N}^{t_0}\cfrac{2}{N}\partial_{\theta,\theta}^2\phi(t',\mu_{t'}^{i})dt'\\
&+\cfrac{2}{\sqrt{N}}\int_{s_N}^{t_0}\partial_\theta\phi(t',\mu_{t'}^{i})dB_{t'}^{\sigma_N(i[N])}.
\end{split}
\end{equation}
We call $\Omega'\subset \Omega$ this set of full probability on which this formula holds.
\subsubsection{Set up}
Fix $\omega\in \Omega'$. We only prove that $F^*(\omega)$ is a subsolution since the proof that it is a supersolution can be done in a same way.
\newline
To lighten notations we will most of the time just write $F^*$ in what follows but we will insist when some quantities depend on $\omega$.
\newline
First, $F^*$ is upper semi continuous and is in $\mathcal F_{t,2\pi}$.
\newline
Let \begin{align*}
A(\omega)&:=\{(\phi,t_0,\theta_0,\delta)\in C^{1,1}_{t,2\pi}\times \R^{+}\times \R\times \R^+\,|\,\\
&(F^*(\omega)-\phi) (t_0,\theta_0)=0 \text{ and }(F^*(\omega)-\phi)(t,\theta)< 0 \text{ for any } (t,\theta)\in B((t_0,\theta_0),\delta)-(t_0,\theta_0)\}.
\end{align*}
From now on, fix $(\phi,t_0,\theta_0,\delta)\in A(\omega)$. Remark that although $A(\omega)$ depends on $\omega$, the quantities $\phi,t_0,\theta_0$ and $\delta$ does not depend on $\omega$ and so are deterministic.
We want to prove that: $$\partial_t \phi(t_0,\theta_0)+(\partial_\theta \phi(t_0,\theta_0))_+\left( I_{1,\delta}[\phi(t_0,.)](\theta_0)+I_{2,\delta}[F^*(t_0,.)](\theta_0)\right)\le 0.$$
In spite of $\phi$ we can consider a function $\phi_1$ such that $\phi_1$ is equal to $\phi$ in $B((t_0,\theta_0),\delta)$, is in $C^{1,1}_{t,2\pi}$, is non decreasing and satisfies that $\phi(t,\theta+2\pi)=\phi(t,\theta)+1$ for all $t>0$ and $\theta\in\R$. If we prove the last inequality with $\phi_1$ instead of $\phi$ it will imply the result for $\phi$ because $\phi$ and $\phi_1$ are equal around $(t_0,\theta_0)$. We shall now use $\phi_1$ in spite of $\phi$ but we still call it $\phi$.
\newline
We focus on the case $\partial_\theta\phi(t_0,\theta_0)>0$. The general case could be deduce from this case by classical viscosity methods that are detailed at the end of the proof of the convergence of the system of particles in the real Dyson case in \cite{bertucci2022spectral}.
\newline
We consider $I_\delta=(\theta_0-\delta,\theta_0+\delta)$. We can choose $\delta$ small enough such that $\phi(t,.)$ is strictly increasing on $I_\delta$ for $t$ near of $t_0$ since $\partial_\theta\phi(t_0,\theta_0)>0$. In what follows we choose such $\delta$. We write $J_\delta(t)=\phi(t,I_{\delta})$ which is an open interval of positive length that contains $\phi(t_0,\theta_0)$ again since $\partial_\theta\phi(t_0,\theta_0)>0$ for $t$ near $t_0$.
\newline
From now on, we consider an increasing subsequence of $(s_N)_{N\in\N}$ that goes to $t_0$ when $N$ goes to $+\infty$. To lighten notation we still call $s_N$ this subsequence.
\newline
We also consider as in Section \ref{subsubsection construction particles} an interval $I$ of length $2\pi$ such that $I_\delta\subset \overset\circ I$.
For every $N>0$ and $\sigma\in\mathfrak{S}_N$, we construct systems of particles $(\mu_t^{i,\sigma_N})_{t\ge s_N, i\in\Z}$ as in Section \ref{subsubsection construction particles}.

\subsubsection{Consequence of the discrete comparison principle}
By definition of $F^*$, for all $N$ we consider an index $i_0(N)$ (which depends on $\omega$) which satisfies: 
\begin{equation}
\left\{
\begin{split}
&\limsup_N \lambda_{s_N}^{i_0(N)}\le \theta_0\\
&\cfrac{i_0(N)}{N}\to F^*(t_0,\theta_0)=\phi(t_0,\theta_0).
\end{split}
\right.
\end{equation}
\newline
We consider $\sigma_N\in\mathfrak{S}_N$ as the permutation that exchanges $1$ and $i_0(N)$ (this permutation depends also on $\omega$.) We shall see in the passage to the limit in Section \ref{subsubsection passage limite} why this choice is important.
\newline
Since $I_\delta\subset \overset\circ I$ we have that for all $i\in\Z$, $\mu_{s_N}^{i,\sigma_N}\le \lambda_{s_N}^i$. We would like to apply the discrete comparison principle but we can not since the particles are not associated with the same Brownian motions due to the permutation $\sigma_N$ that exchanges the first Brownian motion and the $i_0(N)^{th}$ one.
\newline
To apply the discrete comparison principle we define an auxiliary periodic system of $N$ particles $(\lambda_{t}^{i,\sigma_N})_{t\ge s_N, \, i\in\Z}$: 
\begin{equation}
\left\{
\begin{split}
\forall i\in \Z,\,\,\lambda_{s_N}^{i,\sigma_N}=&\lambda_{s_N}^{i}\\
\forall i\in \Z,\,\, d\lambda_t^{i,\sigma_N}=&\cfrac{1}{N}\dis\sum_{k=1}^{N-1}\cotan((\lambda_t^{i,\sigma_N}-\lambda_t^{i+k,\sigma_N})/2)\,dt+\cfrac{2}{\sqrt{N}}\,dB_t^{\sigma_N(i[N])}.
\end{split}
\right.
\end{equation}
By the discrete comparison principle we deduce that for all $t\ge s_N$, for all $i\in\Z$, we have: $\mu_t^{i,\sigma_N}\le \lambda_t^{i,\sigma_N}$.
\newline
So for $N$ large enough we have: $$\phi(s_N,\mu_{s_N}^{i_0(N),\sigma_N})-\phi(t_0,\mu_{t_0}^{i_0(N),\sigma_N})\ge \cfrac{i_0(N)}{N}-\phi(t_0,\lambda_{t_0}^{i_0(N),\sigma_N}), $$ by definition of $\mu_{s_N}^{i}$ and since $\phi$ is non decreasing. By continuity of the trajectories of $(\lambda_.^{i,\sigma_N})_{i\in\Z}$ we get: $$\liminf_N\phi(s_N,\mu_{s_N}^{i_0(N),\sigma_N})-\phi(t_0,\mu_{t_0}^{i_0(N),\sigma_N})\ge F^*(t_0,\theta_0)-\phi(t_0,\theta_0)=0.$$ 
Up to changing the sequence $(s_N)_{N\in\N}$ by another increasing sequence of rational numbers (which depends on $\omega$) that converges to $t_0$ but slow enough, we obtain: 
\begin{equation}
\label{eq: cons disc comp principle}
\liminf_N\cfrac{\phi(s_N,\mu_{s_N}^{i_0(N),\sigma_N})-\phi(t_0,\mu_{t_0}^{i_0(N),\sigma_N})}{t_0-s_N}\ge 0.
\end{equation}

\subsubsection{Interaction between particles}
In this section only, we write $(\mu_{t}^i)_{t\ge s_N,\,i\in\Z}$ instead of $(\mu_{t}^{i,\sigma_N})_{t\ge s_N,\,i\in\Z}$ to lighten notations since the Brownian motions will not appear.
\newline
First we assume that $\mu_{s_N}^i=\gamma_{s_N}^i=(\phi(s_N))^{-1}(i/N)$ if $i/N$ mod 1 is in $J_\delta(s_N)$ mod 1. 
\newline
We separate the sum in two parts:  
\begin{align*}
\cfrac{1}{N}\dis\sum_{k=1}^{N-1}\cotan((\mu_{s_N}^{i_0(N)}-\mu_{s_N}^{i_0(N)+i})/2)&=\cfrac{1}{N}\dis\sum_{1\le i\le N, \, i/N \text{ mod 1 }\notin J_\delta \text{ mod 1}}\cotan((\mu_{s_N}^{i_0(N)}-\mu_{s_N}^{i_0(N)+i})/2)\\
&+\cfrac{1}{N}\dis\sum_{1\le i\le N, \, i/N \text{ mod 1 }\in J_\delta \text{ mod 1}}\cotan((\mu_{s_N}^{i_0(N)}-\mu_{s_N}^{i_0(N)+i})/2)\\
&=:U_N+V_N.
\end{align*}
First we deal with the second term: $V_N$ which is similar with the real Dyson case in \cite{bertucci2022spectral}.
\newline
Let us notice that for $N$ large enough there are particles in $J_\delta$ since the Lebesgue measure of $J_\delta$ is positive and so there will be some $i/N$ in $J_\delta$.
Using the fact that $\cotan$ is non increasing on $(0,\pi)$ we have: \begin{align*}
V_N&=\cfrac{1}{N}\left(\cotan((\mu_{s_N}^{i_0(N)}-\mu_{s_N}^{i_0(N)+1})/2)+\cotan((\mu_{s_N}^{i_0(N)}-\mu_{s_N}^{i_0(N)-1})/2)\right)\\
&+\cfrac{1}{N}\dis\sum_{i\notin \{i_0,i_0+1,i_0-1\}, \, i/N \text{ mod 1 }\in J_\delta \text{ mod 1}}\cotan((\mu_{s_N}^{i_0(N)}-\mu_t^{i_0(N)+i})/2)\\
&\ge \cfrac{1}{N}\left(\cotan((\mu_{s_N}^{i_0(N)}-\mu_{s_N}^{i_0(N)+1})/2)+\cotan((\mu_{s_N}^{i_0(N)}-\mu_{s_N}^{i_0(N)-1})/2)\right)\\
&+\int_{J_\delta-[i_0(N)-2,i_0(N)+1]}\cotan((\mu_{s_N}^{i_0(N)}-(\phi({s_N}))^{-1})(y)/2)dy\\
&\ge \cfrac{1}{N}\left(\cotan((\mu_{s_N}^{i_0(N)}-\mu_{s_N}^{i_0(N)+1})/2)+\cotan((\mu_{s_N}^{i_0(N)}-\mu_{s_N}^{i_0(N)-1})/2)\right)\\
&+\int_{I_\delta-[\mu^{i_0(N)-2},\mu^{i_0(N)+1}]}\cotan((\mu_{s_N}^{i_0(N)}-\theta)/2)\partial_\theta\phi({s_N},\theta)d\theta,
\end{align*}
where we used for the first inequality that: 
\begin{align*}
\cfrac{1}{N}&\sum_{i\notin \{i_0,i_0+1,i_0-1\}, \, i/N \text{ mod 1 }\in J_\delta \text{ mod 1}}\cotan((\mu_{s_N}^{i_0(N)}-\mu_{s_N}^{i_0(N)+i})/2)\\
&=\hspace{1cm}\sum_{i\notin \{i_0,i_0+1,i_0-1\}, \, i/N \text{ mod 1 }\in J_\delta \text{ mod 1}}\int_{(i-1)/N}^{i/N}\cotan((\mu_{s_N}^{i_0(N)}-\mu_{s_N}^{i_0(N)+i})/2)\\
&\hspace{2cm}\ge \int_{J_\delta-[i_0(N)-2,i_0(N)+1]}\cotan((\mu_{s_N}^{i_0(N)}-(\phi({s_N}))^{-1})(y)/2)dy,
\end{align*}
by monotonicity of $(\phi({s_N}))^{-1}$ in the space variable near $\theta_0$.
\newline
Then by regularity of $\phi$, we deduce that: $$\cfrac{1}{N}\left(\cotan((\mu_{s_N}^{i_0(N)}-\mu_{s_N}^{i_0(N)+1})/2)+\cotan((\mu_{s_N}^{i_0(N)}-\mu_{s_N}^{i_0(N)-1})/2)\right)$$ converges to 0 by a Taylor expansion of $\phi^{-1}$ near $\phi(t_0,\theta_0)$ when $N$ goes to $+\infty$. 
\newline
When $N$ goes to $+\infty$, we also have that:
\begin{align*}
\liminf_N &\int_{I_\delta-[\mu^{i_0(N)-2},\mu^{i_0(N)+1}]}\cotan((\mu_{s_N}^{i_0(N)}-\theta)/2)\partial_\theta\phi({s_N},\theta)d\theta\ge\\
& \lim_{\varepsilon\to 0}\int_{I_\delta-[\theta_0-\varepsilon,\theta_0+\varepsilon]}\cotan((\theta_0-\theta)/2)\partial_\theta\phi(t_0,\theta)d\theta:=H_{1,\delta}[\partial_\theta\phi(t_0,.)](\theta_0).
\end{align*}
Hence, we have that: 
\begin{equation}
\label{first interaction}
\liminf_N \cfrac{1}{N}\dis\sum_{1\le i\le N, \, i/N \text{ mod 1 }\in J_\delta \text{ mod 1}}\cotan((\mu_{s_N}^{i_0(N)}-\mu_{s_N}^{i_0(N)+i})/2)\ge H_{1,\delta}[\partial_\theta\phi(t_0,.)](\theta_0).
\end{equation}
Thus, integrating by part yields: $$H_{1,\delta}[\partial_\theta\phi(t_0,.)](\theta_0)=-\cotan(\delta/2)(\phi(t_0,\theta_0+\delta)+\phi(t_0,\theta_0-\delta))+I_{1,\delta}[\phi(t_0,.)](\theta_0).$$
For the first term $U_N$, we shall approximate $F_N(s_N,.)$ by a regular cumulative distribution function.
We fix $N>0$. For all $\varepsilon>0$, we can find a sequence of smooth, strictly non decreasing functions $F_\varepsilon$ such that $(F_\varepsilon)_{\varepsilon>0}$ is a non increasing sequence of functions that converges pointwise to $F_N(s_N,.)$ and such that $F_\varepsilon(\lambda_{s_N}^k-\varepsilon)=F(s_N,\lambda_{s_N}^k)$ for all $k\in\Z$. 
\newline
Since $F_\varepsilon\ge F$ and since $\cotan$ is non increasing on $(-\pi, 0)$ and on $(0,\pi)$, we have that: 
\begin{align*}
\cfrac{1}{N}\dis\sum_{i;\, i/N \text{ mod 1 }\notin J_\delta \text{ mod 1}}\cotan((&\mu_{s_N}^{i_0(N)}-\mu_{s_N}^{i_0(N)+i})/2)=\cfrac{1}{N}\dis\sum_{i;\, i/N \text{ mod 1 }\notin J_\delta \text{ mod 1}}\cotan((\mu_{s_N}^{i_0(N)}-\lambda_{s_N}^{i_0(N)+i})/2)\\
&\ge\cfrac{1}{N}\dis\sum_{1\le i\le N, \, i/N \text{ mod 1 }\notin J_\delta \text{ mod 1}}\cotan((\mu_{s_N}^{i_0(N)}-(\lambda_{s_N}^{i_0(N)+i}-\varepsilon))/2)\\
&\ge\int_{[0,1]-J_\delta}\cotan((\mu_{s_N}^{i_0(N)}-(F_\varepsilon)^{-1}(y)/2)dy\\
&\ge\int_{I-I_\delta}\cotan((\mu_{s_N}^{i_0(N)}-\theta)/2)\partial_\theta F_\varepsilon(\theta)d\theta
\end{align*}
Hence, for all $\varepsilon>0$: 
\begin{align*}
\cfrac{1}{N}\dis\sum_{i;\, i/N \text{ mod 1 }\notin J_\delta \text{ mod 1}}&\cotan((\mu_{s_N}^{i_0(N)}-\mu_{s_N}^{i_0(N)+i})/2)\ge\\
& \int_{I-I_\delta}\cotan((\mu_{s_N}^{i_0(N)}-\theta)/2)\partial_\theta F_\varepsilon(\theta)d\theta:=H_{2,\delta}(\partial_\theta F_\varepsilon)(\mu_{s_N}^{i_0(N)}).
\end{align*}
Thus integrating by part yields: $$
H_{2,\delta}(\partial_\theta F_\varepsilon)(\mu_{s_N}^{i_0(N)})=\cotan(\delta/2)(F_\varepsilon(\mu_{s_N}^{i_0(N)}+\delta)+F_\varepsilon(\mu_{s_N}^{i_0(N)}-\delta))+I_{2,\delta}( F_\varepsilon)(\mu_{s_N}^{i_0(N)}).$$
By monotone convergence theorem we have that for all $N$ large enough: 
\begin{align*}
\cfrac{1}{N}\dis&\sum_{i;\, i/N \text{ mod 1 }\notin J_\delta \text{ mod 1}}\cotan((\mu_{s_N}^{i_0(N)}-\mu_{s_N}^{i_0(N)+i})/2)\ge \\
&\cotan(\delta/2)(F_N(s_N,\mu_{s_N}^{i_0(N)}+\delta)+F_N(s_N,\mu_{s_N}^{i_0(N)}-\delta))+ I_{2,\delta}[F_N(s_N,.)](\mu_{s_N}^{i_0(N)}).
\end{align*}
So by definition of $F^*$ we deduce that: 
\begin{equation}
\label{second interaction}
\begin{split}
\limsup_N \cfrac{1}{N}\dis&\sum_{i;\, i/N \text{ mod 1 }\notin J_\delta \text{ mod 1}}\cotan((\mu_{s_N}^{i_0(N)}-\mu_{s_N}^{i_0(N)+i})/2)\ge\\
& \cotan(\delta/2)(F^*(t_0,\theta_0+\delta)+F^*(t_0,\theta_0-\delta))+I_{2,\delta}[F^*(t_0,.)](\theta_0).
\end{split}
\end{equation}
Hence, if we add \eqref{first interaction} and \eqref{second interaction} we get that: 
\begin{equation}
\begin{split}
\label{formula interaction particles}
\limsup_N \cfrac{1}{N}\dis\sum_{i=1}^{N-1}\cotan((\mu_{s_N}^{i_0(N)}-&\mu_{s_N}^{i_0(N)+i})/2)\ge \\
&\hspace{-4cm}-\cotan(\delta/2)(\phi(t_0,\theta_0+\delta)+\phi(t_0,\theta_0-\delta)-\\
&\hspace{-3cm} F^*(t_0,\theta_0+\delta)-F^*(t_0,\theta_0-\delta))+I_{1,\delta}[\phi(t_0,.)](\theta_0)+I_{2,\delta}[F^*(t_0,.)](\theta_0).
\end{split}
\end{equation}
Now let us explain why we can assume that $\mu_{s_N}^i=\gamma_{s_N}^i=(\phi(s_N))^{-1}(i/N)$ if $i/N$ mod 1 is in $J_\delta(s_N)$ mod 1. 
Since $(t_0,\theta_0)$ is a strict local maximum of $F^*-\phi$ on $B((t_0,\theta_0),\delta)$, for all $\varepsilon>0$ small enough there exists $\gamma>0$ such that: $$\phi(t-\varepsilon, \theta_0+\theta)\ge F^*(t-\varepsilon,\theta_0+\theta)+\gamma,\, \text{ for all } |\theta|\le\delta.$$
Since $(F_N)_N$ is an increasing function that converges toward $F^*$ we can replace $F^*$ with $F_N$ and $\gamma$ by $\gamma/2$ when $N$ is large enough. This implies that for $\delta$ small enough for $N$ large enough for all $i$ such that $i/N$ mod 1 is in $J_\delta(t_0-s_N)$ mod 1, $$(\phi(s_N))^{-1}(i/N)\le (F_N(s_N))^{-1}(i/N)=\lambda_{s_N}^i.$$

\subsubsection{Use of Itô's formula and passage to the limit}
\label{subsubsection passage limite}
Now we use Ito's Formula \eqref{eq: ito formula pas bien} to obtain: 
\begin{equation}
\label{eq: ito formula}
\begin{split}
\phi(t_0,\mu_{t_0}^{i_0(N)})-\phi(s_N,\mu_{s_N}^{i_0(N)})&=\int_{s_N}^{t_0}\partial_\theta\phi(t',\mu_{t'}^{i_0(N)})\cfrac{1}{N}\dis\sum_{k=1}^{N-1}\cotan((\mu_{t'}^{i_0(N)}-\mu_{t'}^{i+k})/2)dt'
\\
&+\int_{s_N}^{t_0}\partial_t\phi(t',\mu_{t'}^{i_0(N)})dt'\\
&+\int_{s_N}^{t_0}\cfrac{2}{N}\partial_{\theta,\theta}^2\phi(t',\mu_{t'}^{i_0(N)})dt'\\
&+\cfrac{2}{\sqrt{N}}\int_{s_N}^{t_0}\partial_\theta\phi(t',\mu_{t'}^{i_0(N)})dB_{t'}^{\sigma_N(i_0[N])}.
\end{split}
\end{equation}
By the choice of $\sigma_N$, $B_{t}^{\sigma_N(i_0[N])}=B_{t}^{1}$ for all $N$.
Now we take $s_N$ that converges sufficiently slow enough toward $t_0$ to be sure that almost surely the two last terms are $o(t_0-s_N)$.
\newline
Let us mention that if we did not use $\sigma_N$ so that the particle $i_0(N)$ is associated to the first Brownian for every $N$, it would not have been clear that the term with the stochastic integral goes to 0 since the index of the Brownian changes also.
\newline
We divide by $t_0-s_N$ the previous expression and we take the $\limsup_N$ to have that: 
\begin{equation}
\begin{split}
\partial_t \phi(t_0,\theta_0)+&(\partial_\theta \phi(t_0,\theta_0))_+
( I_{1,\delta}[\phi(t_0,.)](\theta_0)+I_{2,\delta}[F^*(t_0,.)](\theta_0)-\\&\cotan\left(\cfrac{\delta}{2}\right)(\phi(t_0,\theta_0+\delta)+\phi(t_0,\theta_0-\delta)-F^*(t_0,\theta_0+\delta)-F^*(t_0,\theta_0-\delta)))\le 0
\end{split}
\label{non fini}
\end{equation}
 thanks to the inequality \eqref{eq: cons disc comp principle} and \eqref{formula interaction particles}. 

\subsubsection{Conclusion of the proof}
We nearly have the result we wanted.
We shall use an approximation of $\phi$ to conclude. 
Indeed thanks to Arisawa's lemma (\cite{arisawa2008}, Lemma 2.1) we can find a sequence of smooth function $(\phi_k)_{k\in\N}$ such that for all $k$ we have $\phi_k(t_0,\theta_0)=\phi(t_0,\theta_0)$, $\partial_t\phi_k(t_0,\theta_0)=\partial_t\phi(t_0,\theta_0)$ and $\partial_\theta\phi_k(t_0,\theta_0)=\partial_\theta\phi(t_0,\theta_0)$, for all $(t,\theta)\in B((t_0,\theta_0),\delta)$, we have $F^*(t,\theta)\le \phi_k(t,\theta)\le\phi(t,\theta)$ and $\phi_k(t_0,.)$ is monotone and decreased to $F^*(t_0,.)$. 
\newline
Hence we can apply the inequality $\eqref{non fini}$ to $\phi_k$ instead of $\phi$. 
We get that for all $k\in \N$:
\begin{equation}
\begin{split}
&\partial_t \phi_k(t_0,\theta_0)+(\partial_\theta \phi_k(t_0,\theta_0))_+( I_{1,\delta}[\phi_k(t_0,.)](\theta_0)+I_{2,\delta}[F^*(t_0,.)](\theta_0)-\\&\cotan\left(\cfrac{\delta}{2}\right)(\phi_k(t_0,\theta_0+\delta)+\phi_k(t_0,\theta_0-\delta)-F^*(t_0,\theta_0+\delta)-F^*(t_0,\theta_0-\delta)))\le 0.
\end{split}
\end{equation}
So by the construction of the $\phi_k$ we have that for all $k\in\N$: 
\begin{equation}
\begin{split}
\partial_t \phi(t_0,\theta_0&)+(\partial_\theta \phi(t_0,\theta_0))_+( I_{1,\delta}[\phi_k(t_0,.)](\theta_0)+I_{2,\delta}[F^*(t_0,.)](\theta_0)-\\&\cotan\left(\cfrac{\delta}{2}\right)(\phi_k(t_0,\theta_0+\delta)+\phi_k(t_0,\theta_0-\delta)-F^*(t_0,\theta_0+\delta)-F^*(t_0,\theta_0-\delta)))\le 0.\end{split}
\end{equation}
Then we notice that $\phi_k(t_0,.)-\phi(t_0,.)$ has a maximum in $\theta_0$ and so by the maximum principle for $I_{1,\delta}$ we deduce that: $$ I_{1,\delta}[\phi(t_0,.)](\theta_0)\le I_{1,\delta}[\phi_k(t_0,.)](\theta_0).$$
So thanks to the previous inequality we get that for all $k\in\N$:
\begin{equation}
\begin{split}
\partial_t \phi(t_0,\theta_0&)+(\partial_\theta \phi(t_0,\theta_0))_+ (I_{1,\delta}[\phi(t_0,.)](\theta_0)+I_{2,\delta}[F^*(t_0,.)](\theta_0)-\\&\cotan\left(\cfrac{\delta}{2}\right)(\phi_k(t_0,\theta_0+\delta)+\phi_k(t_0,\theta_0-\delta)-F^*(t_0,\theta_0+\delta)-F^*(t_0,\theta_0-\delta)))\le 0
\end{split}
\end{equation} 
Now we pass to the limit in $k$ and use the fact that $(\phi_k(t_0,.)$ converges pointwise to $F^*(t_0,.)$ to have: 
\begin{equation}
\partial_t \phi(t_0,\theta_0)+(\partial_\theta \phi(t_0,\theta_0))_+\left( I_{1,\delta}[\phi(t_0,.)](\theta_0)+I_{2,\delta}[F^*(t_0,.)](\theta_0)\right)\le 0.
\end{equation} 
Hence $F^*$ is a viscosity subsolution.
\end{proof}

\section{Properties of the Dyson equation on the circle}
Let us recall that the Dyson equation on the circle \eqref{dyson} is: 
\begin{align*}
 \partial_t \mu+\partial_\theta(\mu H[\mu])=0\,\,  \text{in } (0,\infty)\times \T,
\end{align*}
where $H[\mu]=P.V.(\cotan(./2))\ast\mu$ is the Hilbert transform of the measure $\mu$.
We can rewrite this equation:
\begin{equation}
\partial_t \mu+\partial_x \mu H[\mu]+\mu A_0[\mu]=0\,\,  \text{in } (0,\infty)\times \T,
\label{dysonextend}
\end{equation}
where $$A_0[\mu](x)=\frac{1}{2} P.V. \displaystyle\int_{-\pi}^{\pi} \cfrac{\mu(x)-\mu(y)}{\sin^2((x-y)/2)}dy=2P.V. \displaystyle\int_\R \cfrac{\mu(x)-\mu(y)}{(x-y)^2}dy,$$ is called the half Laplacian.
\newline
\tab From now on, we say that $\mu$ is a viscosity solution to \eqref{dyson} if its cumulative distribution function is a viscosity solution to the primitive Dyson equation \eqref{Primitive}.
For $\mu$ a solution to \eqref{dyson}, let us denote $m(t)$ and $M(t)$ as the minimum and the maximal values of $\mu(t,.)$ when they exist: 
$$m(t)=\min_{x\in\T} \mu(t,x)\,\,\text{and}\,\, M(t)=\max_{x\in \T}\mu(t,x).$$
Note that if $\mu$ is a solution to $\eqref{dyson}$ we have a conservation of the mass: $$\int_\T \mu(t,x)dx=\int_\T\mu(0,x)dx,$$ since $$\cfrac{d}{dt}\int_\T \mu(t,x)dx=\int_\T \cfrac{d}{dt} (\mu(t,x))dx=\int_\T\partial_x(\mu(t,.)H[\mu(t,.)](x))dx=0.$$
We shall look at solution $\mu$ such that for all time $t\ge 0$, $\mu(t,.)\in\mesT$.
In the following, we say that a solution to \eqref{dyson} $\mu$ is smooth if for all $t\ge 0$, $\mu(t,.)$ has a density $\mu$ with respect to Lebesgue measure and if $\mu$ is smooth on $\R^+\times\T$.

\subsection{Monotone principle} Given a solution $\mu$ of $\eqref{dyson}$, $m$ and $M$ are monotonic. This result is usual when we deal with equations that involved an operator like $A_0$ that satisfies a maximum principle \cite{cordoba2004maximum,kiselev2022,de2025clogged}.  
\begin{proposition}
\label{prop: monotonie}
Let $\mu$ be a smooth solution to $\eqref{dyson}$. Then $t\mapsto m(t)$ is non decreasing and $t\mapsto M(t)$ is non increasing.
\end{proposition}
\begin{proof}
For all $t>0$, let $x(t)\in\T$ such that $M(t) =\mu(t,x(t))$. We evaluate $\eqref{dysonextend}$ in $(t,x(t))$ and we get that: 
\begin{equation}
\label{ineqmax1}
M'(t)\le-M(t)A_0[\mu(t)](x(t)).
\end{equation}
By the maximum principle of $A_0$ we get that $M'(t)\le 0$ for all $t$. Hence, $M$ is non increasing. By the same argument $m$ is non decreasing.
\end{proof}

\subsection{$L^\infty$ regularization}
\label{se:regula}
We show a $L^\infty$ regularization of the Dyson flow. The approach is similar with the real Dyson case \cite{bertucci2024spectral}. We first state an a priori estimate for a smooth solution to $\eqref{dyson}$. Then, by a regularization argument we prove that this estimate is still true for a viscosity solution to $\eqref{Primitive}$.
\begin{proposition}
\label{prop: estimée infini}
Suppose that $\mu$ is a smooth solution to $\eqref{dyson}$. Then, for all $t>0$:  
$$||\mu(t,.)||_\infty\le \cfrac{1}{2\sqrt{1-\exp(-t)}}.$$
\end{proposition}

\begin{proof}
For all $t>0$, let $x(t)\in\T$ such that $M(t) =\mu(t,x(t))$. We evaluate $\eqref{dysonextend}$ in $(t,x(t))$ and we get that: 
\begin{equation}
\label{ineqmax}
M'(t)\le-M(t)A_0[\mu(t)](x(t)).
\end{equation}
To prove Proposition \ref{prop: estimée infini} we shall obtain a more precise lower bound for the term $A_0[\mu(t)](x(t))$ than just its sign as for the comparison principle. For $\pi>\delta>0$ we have: 
\begin{align*}
A_0[\mu(t)](x(t))\ge& \int_{|x(t)-y|\ge \delta}\cfrac{M(t)-\mu(t,y)}{2\sin^2((x(t)-y)/2))} \,dy \\
\ge& M(t)\int_{|x(t)-y|\ge\delta}\cfrac{1}{2\sin^2((x(t)-y)/2)} \,dy-\int_{|x(t)-y|\ge\delta}\cfrac{\mu(t,y)}{2\sin^2(\delta/2)} dy\\
&\ge 2M(t)\cotan\left(\cfrac{\delta}{2}\right)-\cfrac{1}{2\sin^2(\delta/2)}. 
\end{align*}  
We optimize this inequality in $\delta$. We find that the optimal bound is obtained for $\delta$ such that $$2M(t)=\cotan\left(\cfrac{\delta}{2}\right).$$ We notice that for the choice of this delta we have that: $$\cfrac{\cos(\delta/2)}{\sin(\delta/2)}=2M(t)$$ which yields: $$4M(t)^2=\cfrac{1-\sin^2(\delta/2)}{\sin^2(\delta/2)}\,\,,\,\, \cfrac{1}{\sin^2(\delta/2)}=4M(t)^2+1.$$
Hence we deduce that a lower bound for $A_0[\mu(t)](x(t))$ is: $$ A_0[\mu(t)](x(t))\ge 2M(t)^2-\cfrac{1}{2}.$$
Hence we have that: $$M'(t)\le -2M(t)\left(M(t)^2-\cfrac{1}{4}\right)\le -2M(t)\left(M(t)-\cfrac{1}{2}\right)\left(M(t)+\cfrac{1}{2}\right).$$
Let $t^*=\inf\left\{t;\, M(t)\le \cfrac{1}{2}\right\}$ which is possibly infinite. For $t>t^*$, by Proposition $\ref{prop: monotonie}$ we deduce: $$M(t)\le \cfrac{1}{2}\le \cfrac{1}{2\sqrt{1-\exp(-t)}}.$$
So, it remains to prove the result for $t\le t^*$. Let $t<t^*$, we have that: 
 $$\cfrac{M'(t)}{M(t)\left(M(t)-\cfrac{1}{2}\right)\left(M(t)+\cfrac{1}{2}\right)}\le -2.$$
Thus, integrating this relation between $0$ and $t$ yields: $$-4\ln\left(\cfrac{M(t)}{M(0)}\right)+2 \ln\left(\cfrac{M(t)-\cfrac{1}{2}}{M(0)-\cfrac{1}{2}}\right)+2\ln\left(\cfrac{M(t)+\cfrac{1}{2}}{M(0)+\cfrac{1}{2}}\right)\le-2t.$$
Taking the exponential we get: 
$$\left(1-\cfrac{1}{4M(t)^2}\right)^2=\cfrac{\left(M(t)^2-\cfrac{1}{4}\right)^2}{M(t)^4}\le \exp(-2t)\cfrac{\left(M(0)^2-\cfrac{1}{4}\right)^2}{M(0)^4}\le \exp(-2t).$$
We conclude that: $$M(t)\le\cfrac{1}{2\sqrt{1-\exp(-t)}}.$$
Finally if $t^*>0$ we have that for all $t<t^*$: $$M(t^*)\le M(t)\le\cfrac{1}{2\sqrt{1-\exp(-t)}},$$ by taking the limit with $t$ converges to $t^*$ we get that: $$M(t^*)\le\cfrac{1}{2\sqrt{1-\exp(-t^*)}}.$$
\end{proof}

We shall prove a regularization lemma to justify that this estimate is actually true for non smooth solution  to \eqref{dyson}.
\begin{lemma}
\label{lemma: Regularization}
Let $\varepsilon>0$ and $\mu_0\in C^\infty(\T)$ be a probability density. Then there exists a (unique) smooth solution to:
\begin{equation}
\label{equationsmooth}
\partial_t \mu-\varepsilon\partial_{xx} \mu +\partial_x(\mu H[\mu])=0\, \in (0,\infty)\times \T,
\end{equation}
which satisfies the initial condition $\mu(0,.)=\mu_0$.
\end{lemma}

\begin{proof}
Fix a time horizon $T>0$. We start by assuming that a smooth solution exists and we shall obtain an a priori estimate. In what follows $C$ will denote a constant depending only on $\varepsilon$, $m_0$ and $T$.  First, since the term in $\varepsilon$ does not perturb the proof of the $L^\infty$ a priori estimate above, for all $t\ge 0$ $||\mu(t,.)||_\infty \le ||\mu(0.)||_\infty.$  It implies that for all $p\in[1,+\infty]$, $\mu\in L^\infty_t(L^p_x)$. Multiplying \eqref{equationsmooth} by $\mu$ and integrating over space and time we get: 
$$\int_\T\mu(t,.)^2 +\varepsilon \int_0^t \int_\T(\partial_x\mu)^2ds=\int_0^t\int_\T (\partial_x \mu) \mu H[\mu]+\int_\T\mu(0,.)^2.$$
Using Cauchy–Schwarz inequality, the fact that $\mu\in L_t^\infty(L^2_x)$ and the isometry property of the Hilbert transform, we deduce that $\mu H[\mu]\in L^\infty_t(L^1_x)$. Moreover let us notice that for all $0<\alpha<1$ there exists a constant $C_\alpha>0$ such that for all $\mu\in C^\alpha$, $H[\mu]\le C_\alpha ||\mu||_{C^\alpha}$. If $\alpha<1/2$, by using Sobolev's embedding we can bound the previous inequality by $||\mu||_{H^1}$. 
We can bound: 
\begin{align*}
\int_0^t\int_\T (\partial_x \mu) \mu H[\mu]&\le \cfrac{\varepsilon}{2}\int_0^t\int_\T(\partial_x \mu)^2+\cfrac{1}{2\varepsilon}\int_0^t\int_\T (\mu H[\mu])^2\\
&\le \cfrac{\varepsilon}{2}\int_0^t\int_\T(\partial_x \mu)^2+\cfrac{C}{2\varepsilon}\int_0^t||\mu||_{H_x^1}\\
&\le \cfrac{\varepsilon}{2}\int_0^t\int_\T(\partial_x \mu)^2+\cfrac{C\sqrt{T}}{2\varepsilon}\sqrt{\int_0^t \int_\T (\partial_x\mu)^2}+C,
\end{align*}
where we used the Cauchy–Schwarz inequality for the last bound.
So there exists $C>0$ such that:  $$\int_\T\mu(t,.)^2+\cfrac{\varepsilon}{2} \int_0^t \int_\T(\partial_x\mu)^2ds-\cfrac{C}{2\varepsilon}\sqrt{\int_0^t \int_\T (\partial_x\mu)^2}\le C.$$
From this bound we deduce that $\mu\in L^2_t(H^1_x)$.
\newline
Further regularity can be obtained by looking at the equation satisfied by $w=\partial_x\mu$.
Multiplying this relation by $w$ and integrating yields
$$ \int_\T w(t,.)^2 +\varepsilon \int_0^t\int_\T(\partial_x w)^2=-\int_0^t\int_\T w^2H[w]-\int_0^t\int_\T\partial_x w[wH[\mu]-\mu H[w]]+\int_\T w(0,.)^2.$$
Arguing as in the previous case to bound, we deduce that $w\in L^\infty_t(L^2)$ and $w\in L^2_t(H^1_x)$. Further regularity can then be obtained in a similar fashion by boot strapping techniques.
\newline
Once regularity is obtained, the existence of a solution to \eqref{equationsmooth} with this regularity is
classical and we do not detail it here. Furthermore, in this situation, the uniqueness of such
a solution is immediate. 
\end{proof}
Let us notice that for all $\varepsilon>0$, the unique solution to \eqref{equationsmooth} satisfies the same a priori estimate obtained in Proposition \ref{prop: estimée infini}.
\newline
\tab Once such a priori estimates which are independent of $\varepsilon>0$ and of the initial data have been obtained, it is standard to pass to the the limit and propagate it to the unique viscosity solution to the limiting equation when $\varepsilon$ goes to 0 (see for instance \cite{crandall1992,bertucci2024spectral,biler2008nonlinear}).

\begin{proposition}
\label{Bound Linfini}
Let $\mu_0\in\mesT$ and $\mu$ be the unique solution to $\eqref{dyson}$ with initial data $\mu_0$. Then for all $t>0$, $\mu(t)$ has a density with respect to the Lebesgue measure bounded by: $$||\mu(t,.)||_\infty\le\cfrac{1}{2\sqrt{1-\exp(-t)}}.$$
\end{proposition}

We directly obtain a regularization of the $L^p$ norm of a solution to the Dyson equation.
\begin{corollary}
\label{coro:L^p reg}
Let $\mu_0\in\mesT$ and $\mu$ be the unique solution to $\eqref{dyson}$ with initial data $\mu_0$. Then for all $t>0$, $\mu(t)$ has a density with respect to the Lebesgue measure which is in $L^p(\T)$: $$||\mu(t,.)||_{L^p(\T)}\le\cfrac{\pi}{\sqrt{1-\exp(-t)}}.$$
\end{corollary}

\subsection{Decay of $L^p$ norms} We prove the following statement. 
\begin{proposition}
\label{prop: Lp decay}
For $\mu(0)\in\mesT$, the unique solution to \eqref{dyson} with initial condition $\mu(0)$ is such that for any $0\le t\le s$, $1\le p\le +\infty$ $$||\mu(t,.)||_p\ge ||\mu(s,.)||_p.$$
\end{proposition}
\begin{proof}
The cases $p=1$ and $p=+\infty$ have already be treated in previous sections. Let $1<p<+\infty$. We only prove the statement for smooth solutions, the general result can be obtained by approximation as in the previous section. We compute: 
\begin{align*}
\cfrac{d}{dt}||\mu(t)||_p^p=&p\int_\T\mu(t,\theta)^{p-1}\partial_t\mu(t,\theta)d\theta\\
&=-p\int_\T\mu(t,\theta)^{p-1}\partial_\theta(\mu H[\mu])(t,\theta)d\theta\\
&=p\int_\T\partial_\theta(\mu(t,\theta)^{p-1})(\mu H[\mu])(t,\theta)d\theta\\
&=(p-1)\int_\T\partial_\theta(\mu(t,\theta)^p)H[\mu(t,.)](\theta)d\theta\\
&=-(p-1)\int_\T\mu(t,\theta)^p A_0[\mu(t,.)](\theta)d\theta\\
&=-\cfrac{p-1}{2}\int_{-\pi}^\pi \mu(t,\theta)^p\int_{-\pi}^{\pi}\cfrac{\mu(t,\theta)-\mu(t,\theta')}{\sin^2((\theta-\theta')/2)}d\theta'd\theta\\
&=-\cfrac{p-1}{4}\int_{-\pi}^\pi \int_{-\pi}^{\pi}\cfrac{(\mu(t,\theta)-\mu(t,\theta'))(\mu(t,\theta)^p-\mu(t,\theta')^p)}{\sin^2((\theta-\theta')/2)}d\theta'd\theta.
\end{align*}
The last term being clearly non-positive since $\mu\ge 0$, the result follows. 
\end{proof}

\subsection{Free entropy}
Let us introduce the free entropy of the physical system associated to the Dyson flow on the circle originally introduced by Voiculescu in the real case \cite{voiculescu1993analogues} $\mathcal I: \mesT\to\R\cup\{+\infty\}$ defined by $$\mathcal I(\mu)=-\int_{\T^2} \ln(|\sin((x-y)/2)|)\mu(dx)\mu(dy)-\ln(2)=-\int_\T (\ln(|\sin(./2)|)\ast \mu)(y)\mu(dy)-\ln(2),$$ where $\ln(|\sin(./2)|)\ast \mu$ is understood as the convolution between the distributions
\newline
 $P.V.(\ln(|\sin(./2)|))$ and $\mu$.
The terminology of free entropy makes sense with the fact that in the language of large deviations theory $\mathcal I$ is the good rate function for the empirical mean of eigenvalues of matrices distributed in the circular unitary ensemble \cite{hiai2000large} (and \cite{arous1997large} for the gaussian unitary ensemble) as the usual entropy is the good rate function for large deviation of the empirical mean of independent and identically distributed random variables in Sanov's theorem.
\newline
From now on let us write $\mu_{\text{Unif}}(dx)=dx/(2\pi)$ for the uniform measure on the circle.
\newline
We shall mention some basic properties of $\mathcal I$ and prove later on that it is
decreasing and continuous along the Dyson flow.

\begin{lemma}
We have the following properties for $\mathcal I$.
\begin{enumerate}
\item{$\mathcal I$ is lower semi continuous for the narrow topology, bounded from below and so admits a minimum on $\mesT$. Moreover for all $M\in\R$, $\{\mathcal I\le M\}$ is a compact subset of $\mesT$.}
\item{$\mathcal I(\mu_{\text{Unif}})=0$}
\item{For all $\mu\in\mesT$, 
\begin{equation}
\label{fourier}
\mathcal I(\mu)=\cfrac{1}{2}\sum_{n\in\Z-\{0\}}\cfrac{|c_n(\mu)|^2}{|n|},
\end{equation}
where $c_n(\mu)$ is the $n^{th}$ Fourier coefficient of $\mu$ defined by $c_n(\mu)=\int_\T \exp(-int)\mu(dt).$}
\item{Take $\mu\in \mesT$, then $\mathcal I(\mu)<+\infty$ if and only $\mu\in H^{-1/2}(\T)$.}
\item{$\mathcal I$ is non negative and admits a unique minimiser which is $\mu_{\text{Unif}}$.}
\item{Let $\mu,\nu\in\mesT$, then $\mathcal I(\mu-\nu)\ge 0$ with strict inequality if $\mu\ne \nu$. As a consequence $\mathcal I$ is strictly convex on $\mesT$}
\end{enumerate}
\end{lemma}
\begin{proof}
\begin{enumerate}
\item{We use the usual method in large deviation to show that a such function is a good rate function \cite{arous1997large,dembo2009large,hiai2000large}. Let $K>0$ and let us define $\mathcal I_K:\mesT\to \R$ by $$\mathcal I_K(\mu)=\int_\T \min(-\ln(|\sin((x-y)/2)|),K)\mu(dx)\mu(dy)-\ln(2).$$ Since $(x,y)\to\min(-\ln(|\sin(x-y)/2)|),K)$ is continuous and bounded for all $K>0$, we that for all $K>0$ $\mathcal I_K$ is continuous. Since $\sup_{K>0}\mathcal I_K=\mathcal I$, $\mathcal I$ is lower semi continuous as a supremum of lower semi continuous function. Moreover we have that $\mathcal I\ge -\ln(2)$. Let us remark that $\mesT$ is compact since $\T$ is compact. Hence, since $\mathcal I$ is lower semi continuous, it admits a minimum of $\mesT$ and for all $M\in\R$ $\{\mathcal I\le M\}$ is a closed subset of $\mesT$ and so is compact.}
\item{In order to compute $\mathcal I(\mu_{\text{Unif}})$, we first compute $\ln(|\sin(./2)|)\ast \mu_{\text{Unif}}$
\begin{lemma}
\label{lemma: conv}
We have the following relation: $$(\ln(|\sin(./2)|)\ast \mu_{\text{Unif}})(\theta)=-\ln(2),\,\, \forall \theta\in\T.$$
\end{lemma}
\begin{proof}
First we notice that in the sense of distributions $$(\ln(|\sin(./2)|)\ast \mu_{\text{Unif}})'=\cfrac{1}{2}\cotan(./2)\ast\mu_{\text{Unif}}.$$But for all $\theta\in\T$, $$\left(\cfrac{1}{2}\cotan(./2)\ast\mu_{\text{Unif}}\right)(\theta)=\cfrac{1}{4\pi}\,\,P.V\int_\T\cotan((\theta-\tilde\theta)/2)d\tilde\theta=0, $$ by antisymmetry of the cotangent function. Hence there exists a constant $C\in\R$ such that for all $\theta \in\T$ $$(\ln(|\sin(./2)|)\ast \mu_{\text{Unif}})(\theta)=C.$$ We evaluate this identity in 0 to get that $$\cfrac{1}{2\pi}\int_\T\ln(|\sin(\theta/2)|)d\theta=\cfrac{2}{\pi}\int_0^{\pi/2}\ln(\sin(\theta))d\theta=C.$$
Let $I=\int_0^{\pi/2}\ln(\sin(\theta))d\theta$. We compute this integral:
\begin{align*}
2I&=\int_0^{\pi/2}\ln(\sin(\theta))d\theta+\int_0^{\pi/2}\ln(\cos(\theta))d\theta\\
&=\int_0^{\pi/2}\ln(\sin(2\theta)/2)d\theta\\
&=\int_0^{\pi/2}\ln(\sin(\theta))d\theta-\cfrac{\pi}{2}\ln(2)\\
&=I-\cfrac{\pi}{2}\ln(2)
\end{align*}
We get that for all $\theta\in\T$, $$(\ln(|\sin(./2)|)\ast \mu_{\text{Unif}})(\theta)=-\ln(2)$$
\end{proof}
Since $$\mathcal I(\mu_{\text{Unif}})=-\int_\T (\ln(|\sin(./2)|)\ast \mu)(y)\mu(dy)-\ln(2),$$ we get that $\mathcal I(\mu_{\text{Unif}})=0$.}
\item{Let us notice that:
\begin{align*}
\int_\T (\ln(|\sin(./2)|)\ast \mu)(y)\mu(dy)&=\sum_{n\in\Z}c_n(\ln(|\sin(./2)|)\ast \mu)\overline{c_n(\mu)}\\
&=\sum_{n\in\Z}c_n(\ln(|\sin(./2)|))|c_n(\mu)|^2.
\end{align*}
It remains to compute for all $n\in\Z$ $c_n(\ln(|\sin(./2)|))$. Let us remark that by Lemma \ref{lemma: conv}, $c_0(\ln(|\sin(./2)|))=-\ln(2)$. Moreover, for $n\ne 0$ we have: $$c_n(\ln(|\sin(./2)|))=\cfrac{c_n\left(\cfrac{1}{2}\cotan(./2)\right)}{in}.$$
For $n\ne 0$, 
\begin{align*}
c_n\left(\cfrac{1}{2}\cotan(./2)\right)&=P.V.\int_\T\exp(-in t)\cfrac{\cotan(t/2)}{2}\cfrac{dt}{2\pi}\\
&=P.V.\int_\T\exp(-in t)\sum_{k\in\Z}\cfrac{1}{t-2k\pi} \cfrac{dt}{2\pi}\\
&=\sum_{k\in\Z}P.V.\int_{2k\pi}^{2(k+1)\pi}\cfrac{\exp(-int)}{t}\cfrac{dt}{2\pi}\\
&=P.V.\int_\R\cfrac{\exp(-int)}{t}\cfrac{dt}{2\pi}=-\cfrac{i}{2}\, \text{sign}(n),
\end{align*} 
where we used the Euler's identity for cotangent and the fact that the Fourier transform of the distribution $P.V.(1/x)$ is $-i\pi \,\text{sign}(\xi)$.
\newline
Hence, we get that 
\begin{align*}
\int_\T (\ln(|\sin(./2)|)\ast \mu)(y)\mu(dy)&=\sum_{n\in\Z}c_n(\ln(|\sin(./2)|))|c_n(\mu)|^2\\
&=-\ln(2)|c_0(\mu)|^2-\cfrac{1}{2}\sum_{k\in\Z-\{0\}}\cfrac{|c_k(\mu)|^2}{|k|}.
\end{align*}
Since for $\mu\in\mesT$ we have $c_0(\mu)=1$, for $\mu\in\mesT$, 
$$\mathcal I(\mu)=\cfrac{1}{2}\sum_{k\in\Z-\{0\}}\cfrac{|c_k(\mu)|^2}{|k|}.$$ }
\item{The relation \eqref{fourier} directly implies the result by definition of $H^{-1/2}(\T)$. }
\item{Thanks to \eqref{fourier} we see that $\mathcal I$ is non negative and moreover if $\mu\in\mesT$ is such that $\mathcal I(\mu)=0$ we get that for all $n\ne 0$ $c_n(\mu)=0$. By uniqueness of the Fourier coefficient, $=\mu=\mu_{\text{Unif}}$.
}
\item{Actually, the formula \eqref{fourier} is true for $\mu-\nu$ since $\mu$ and $\nu$ have the same mass ($c_0(\mu)=c_0(\nu)$). Hence, we directly have that $\mathcal I(\mu-\nu)\ge 0$. By using the same argument that shows that the unique minimizer is $\mu_{\text{Unif}}$, we get that $\mathcal I(\mu-\nu)=0$ if and only if $\mu-\nu=0$. The strict convexity is a direct consequence of this property and that $\mathcal I$ is a quadratic form. }
\end{enumerate}
\end{proof}

\begin{proposition}
\label{prop:freeentrop}
Let $(\mu_t)_{t\ge0}$ be a solution to \eqref{dyson} such that $\mathcal I(\mu_0)<+\infty$. Then for any $0\le t'\le t$:
\begin{equation}
\label{freentropy}
\mathcal I(\mu_t)=\mathcal I(\mu_{t'})-\int_{t'}^t\int_\T (H[\mu_s](x))^2\mu_s(dx)ds.
\end{equation}
\end{proposition}
\begin{proof}
First let us suppose that $\mu_t$ is a smooth solution to \eqref{dyson}. We compute:
\begin{align*}
\cfrac{d}{dt}\,\mathcal I(\mu_t)&=-2\int_\T\ln(|\sin((x-y)/2)|)\mu_t(dy)\partial_t\mu_t(dx)\\
&=-2\int_\T (\ln(|\sin(./2)|)\ast \mu_t)(x)\partial_t\mu_t(dx)\\
&=+2\int_\T (\ln(|\sin(./2)|)\ast \mu_t)\partial_x(H[\mu_t]\mu_t)\\
&=-2\int_\T \partial_x(\ln(|\sin(./2)|)\ast \mu_t)(x)H[\mu_t](x)\mu_t(dx)\\
&=-\int_\T (H[\mu_t](x))^2\mu_t(dx)
\end{align*}
We deduce that: $$\mathcal I(\mu_t)=\mathcal I(\mu_{t'})-\int_{t'}^t\int_\T (H[\mu_s](x))^2\mu_s(dx)ds.$$
To prove that the identity \eqref{freentropy} is still true for non smooth solution we can do the same approach as in \cite{bertucci2024spectral} in the real Dyson case by first proving this identity if $\mu(0,.)$ is bounded by regularisation of the kernel $-\ln(\sin(./2)$ by convolution and using the continuity of the Hilbert transform in all the $L^p$ for $1\le p<+\infty$. For the general case, thanks to the $L^\infty$ regularisation result we have that for all $t>0$, $\mu(t,.)$ is bounded. So we can apply the result for $0<t'\le t$ and pass to the limit paying attention to the fact that $\mathcal I$ is just lower semi continuous and not continuous. 
\end{proof}
\begin{corollary}
Let $\mu_0\in\mesT$ be such that $\mathcal I(\mu_0)<\infty$. Then $(\mathcal I(\mu_t))_{t\ge0}$ is a continuous non increasing function.
\end{corollary}
As an application we use this property to prove the convergence towards the uniform measure on the circle of a solution to the Dyson equation. Let us mention that it is already known in the literature using a more computational approach.  Indeed in \cite{cepa2001brownian} it has been proved that if $\mu$ is a solution to the Dyson equation then $\mu(t,.)$ converges in law towards $\mu_{\text{Unif}}$ when $t$ goes to $\infty$. The proof is based on the computation of the Fourier coefficients of a solution to $\eqref{dyson}$ and finding a recursive equation satisfied by these coefficients and then finding a unique solution at the limit of this recursive equation.

\begin{proposition}
Let $(\mu_t)_{t\ge0}$ be a solution to \eqref{dyson} such that $\mu_0\in\mesT$. Then $\mu_t\underset{t\to\infty}{\longrightarrow}\mu_{\text{Unif}}$ for the narrow topology. 
\end{proposition}

\begin{proof}
First, since we are interested in the long time behaviour of the solution $(\mu_t)_{t\ge0}$, we can suppose without loss of generality that $\mathcal I(\mu_0)<+\infty$. Indeed, by Proposition \ref{Bound Linfini}, for $t>0$, $\mu_t$ is in $L^\infty(\T)$ and so $\mathcal I(\mu_t)<+\infty$.
\newline
\tab Let us recall that since $\mesT$ is compact for the narrow topology, $(\mu_t)_{t\ge0}$ is tight. So, it remains to prove that the only accumulative point of $(\mu_t)_{t\ge0}$ is $\mu_{\text{Unif}}$. To lighten the notations let us assume that $\mu_t\underset{t\to\infty}{\longrightarrow}\mu\in\mesT$ for the narrow topology instead of considering an extraction.
\newline
By Proposition $\ref{prop:freeentrop}$, for all $t\ge 0$, $$\mathcal I(\mu_t)=\mathcal I(\mu_{0})-2\int_{0}^t\int_\T (H[\mu_t](x))^2\mu_t(dx).$$
Let us recall the Cotlar's identity on the circle. For $u\in L^2(\T)$, if we define $H[u](\theta):=P.V. \int_\T\cotan((\theta-\theta')/2)u(\theta')d\theta'/(2\pi)$ and $c_0(u):=\int_\T u(\theta)\,d\theta/(2\pi)$ we have $$H[u]^2=u^2-c_0(u)^2+2H[uH[u]].$$
Evaluation this identity in $2\pi\mu(t,.)$ for $t>0$ we get that: 
\begin{equation}
\label{integration}
4\pi^2 H[\mu_t]^2=4\pi^2\mu_t^2-1+8\pi^2 H[\mu_tH[\mu_t]].
\end{equation}
Multiplying by $\mu_t$ and integrating in space yields: $$4\pi^2 \int_\T H[\mu_t]^2\mu_t(d\theta)=4\pi^2\int_\T\mu_t(\theta)^3 d\theta-1+8\pi^2 \int_\T H[\mu_tH[\mu_t]]\mu_t(d\theta).$$
Using the usual antisymmetric property of the Hilbert transform, for $f$ and $g$ real valued and in $L^2(\T)$ we have:  $$\int_\T H[f]gd\theta=-\int_\T f H[g] d\theta.$$ 
Using this identity in \eqref{integration}, we get that for all $t>0$: 
$$\int_\T H[\mu_t]^2\mu_t(d\theta)=\cfrac{1}{3}\left(\int_\T\mu_t(\theta)^3d\theta-\cfrac{1}{4\pi^2}\right).$$
Hence for all $t\ge 0$, we get: $$\mathcal I(\mu_t)=\mathcal I(\mu_{0})-\cfrac{2}{3}\left(\int_{0}^t\left[\int_\T \mu_t(\theta)^3d\theta-\cfrac{1}{4\pi^2}\right]dt\right).$$
Since $(\mathcal I(\mu_t))_{t\ge 0}$ is non negative we deduce that $$\int_{0}^{+\infty}\left[\int_\T \mu_t(\theta)^3d\theta-\cfrac{1}{4\pi^2}\right]dt<+\infty.$$
By Proposition $\ref{prop: Lp decay}$ for $p=3$ the integrand in time is non increasing in time. This implies that 
\begin{equation}
\label{eq:mino}
\int_\T\mu_t(\theta)^3 d\theta\underset{t\to\infty}{\longrightarrow}\cfrac{1}{4\pi^2}.
\end{equation}
Let us notice that for $\nu$ a probability density on $\T$, Holder's inequality implies $$1=\int_\T\nu(\theta)d\theta\le \left(\int_\T\nu^3(\theta)d\theta\right)^{1/3}(2\pi)^{2/3}.$$
So, for $\nu$ a probability density on $\T$, $$\int_\T\nu^3(\theta)d\theta\ge\cfrac{1}{4\pi^2}$$ with equality if and only if $\nu=1/(2\pi)$. We shall use this property to prove that $\mu$ is the uniform measure on $\T$ since the lower bound is exactly the limit we get in \eqref{eq:mino}. 
\newline
Since $(\mu_t)_{t\ge 1}$ is bounded in $L^3(\T)$ by Proposition \ref{coro:L^p reg}, we can consider up to extraction a weak limit in $\sigma(L^3(\T),(L^3(\T))'=L^{3/2}(\T))$ where the $(L^3(\T))'$ is the topological dual of $L^3(\T)$. Let us write $\tilde \mu$ this limit and forget the extraction to lighten notations again. More exactly we have that for for all $\phi\in L^{3/2}(\T)$, $$\int_\T\phi(\theta)\mu_t(\theta)d\theta\underset{t\to\infty}{\longrightarrow}\int_\T\phi(\theta)\tilde \mu(\theta)d\theta.$$
If we take $\phi=1\in L^{3/2}(\T)$ and $\phi=\mathds 1_{\tilde\mu<0}\in L^{3/2}(\T)$ we get that $\tilde\mu$ is a probability density on $\T$ since for all $t$, $\mu_t$ is.
Moreover since $C(\T)\subset L^{3/2}(\T)$, we get that $(\mu_t)_{t\ge1}$ converges for the narrow topology to $\tilde \mu(\theta)d\theta$. By uniqueness of this limit we get $\mu=\tilde \mu(\theta)d\theta$.
Finally let us recall that since $(\mu_t)_{t\ge 1}$ converges weakly in $L^3(\T)$ to $\tilde\mu$, we have: $$||\tilde\mu||_{L^3}^3\le\liminf_{t\ge 1}||\mu_t||_{L^3}^3=\cfrac{1}{4\pi^2},$$ by $\eqref{eq:mino}$. By the Holder equality case we get that $\tilde\mu(\theta)=1/(2\pi)$. Hence, we get that $\mu=\mu_{\text{Unif}}$ is the unique limit for the narrow topology to $(\mu_t)_{t\ge 1}$ which gives the result.
\end{proof}

\subsection{Quantitative convergence to the equilibrium}

As explained, it is known that if $\mu$ is a solution to the Dyson equation then $\mu(t,.)$ converges in law towards $\mu_{\text{Unif}}$ when $t$ goes to $\infty$ \cite{cepa2001brownian}. Since we know by Section $\ref{se:regula}$ that $\mu(t,.)$ has a density for $t>0$, it is quite natural to look at the convergence in $L^p(\T)$ norm for $p\in[1,+\infty]$ of these densities towards the uniform density on the circle (which implies the convergence in law).

\subsubsection{Convergence of the maximum of a solution and convergence in $L^p$ for $p\in[1,\infty)$}
\label{subsubsection:maxconv}
We first prove that the maximum of a solution of the Dyson equation converges to $\frac{1}{2\pi}$. This proof is based on the so-called Constantin-Vicol lemma for the fractional Laplacian \cite{constantin2012nonlinear}.
\begin{lemma}
\label{lemma:boundlaplacianfractional}
Let $f\in C(\T)$ be a non negative function such that $f(x_0)=M$ is the maximum of $M$ and $f(\tilde x_0)=m$ is the minimum of $f$. Then we have the following bounds: 
\begin{align*}
A_0(f)(x_0)\ge&\cfrac{2}{\pi^2}\left(2\pi M-||f||_{L^1(\T)}\right)\\
A_0(f)(\tilde x_0)\le&\cfrac{2}{\pi^2}\left(2\pi m-||f||_{L^1(\T)}\right).
\end{align*}
\end{lemma}

\begin{proof}
We only give a proof for the first bound since the second one is exactly similar. 
\newline 
Start by the Euler identity of $1/\sin^2$
\begin{equation}
\begin{split}
A_0(f)(x_0)&=\cfrac{1}{2}\int_{[x_0-\pi,x_0+\pi]} \cfrac{f(x_0)-f(y)}{\sin^2((x_0-y)/2)}dy\\
&=2\int_{[x_0-\pi,x_0+\pi]} \left[f(x_0)-f(y)\right]\sum_{k\in\Z}\cfrac{1}{(x_0-y-2k\pi^2)}dy.
\end{split}
\end{equation}
Since all the quantities in the integrand are non negative, Fubini-Tonelli's theorem gives
\begin{equation}
\begin{split}
A_0(f)(x_0)&=2\int_{[x_0-\pi,x_0+\pi]} \left[f(x_0)-f(y)\right]\sum_{k\in\Z}\cfrac{1}{(x_0-y-2k\pi)^2}dy\\
&\ge 2\int_{[x_0-\pi,x_0+\pi]} \cfrac{f(x_0)-f(y)}{(x_0-y)^2}dy.
\end{split}
\end{equation}
If $y\in [x_0-\pi,x_0+\pi]$ then $|y-x_0|\le \pi$ and so we have:
\begin{equation}
A_0(f)(x_0)\ge \cfrac{2}{\pi^2}\int_{[x_0-\pi,x_0+\pi]} \left[f(x_0)-f(y)\right]dy=\cfrac{2}{\pi^2} \left[2\pi M-||f||_{L^1(\T)}\right]
\end{equation}
\end{proof}

\begin{proposition}
\label{Prop:longtimemax}
Let $\mu$ be a solution of \eqref{dyson} then for all $t\ge 0$ 
\begin{equation}
\label{Linftybound}
M(t)\le \frac{1}{2\pi}\cfrac{1}{1-\exp(-2t/\pi^2)}.
\end{equation}
\end{proposition}

\begin{proof}
We only give an a priori estimate for a smooth solution of \eqref{dyson}. To justify that this estimate is true for every solution of \eqref{dyson} we can do the exact same proof as in Section \ref{se:regula}.
\newline
For $t\ge 0$ let $x(t)$ such that $M(t)=\mu(t,x(t))$. Evaluating the Dyson equation \eqref{dyson} in $(t,x(t))$ we get $$M'(t)\le -M(t)A_0[\mu(t,.)](x(t)).$$
Using Lemma \ref{lemma:boundlaplacianfractional} we obtain 
\begin{equation}
\label{eq:majoration1}
M'(t)\le -\cfrac{2M(t)}{\pi^2}\, \left[2\pi M(t)-1\right]=-\cfrac{4M(t)}{\pi}\left[M(t)-\cfrac{1}{2\pi}\right].
\end{equation}
We recall that for all $t\ge 0$ we have $M(t)\ge \frac{1}{2\pi}$. So if there exists $t^*\ge 0$ such that $M(t)=\frac{1}{2\pi}$ then $M(t)=\frac{1}{2\pi}$ for $t\ge t^*$ by the monotone principle and so the inequality \eqref{Linftybound} is true for $t\ge t^*$. So it remains to prove the result when $M(t)>\frac{1}{2\pi}$. In this case we can divide and integrate the equation \eqref{eq:majoration1} to obtain:
$$2\pi\left[\ln\left(\cfrac{M(t)-\frac{1}{2\pi}}{M(0)-\frac{1}{2\pi}}\right)-\ln\left(\cfrac{M(t)}{M(0)}\right)\right]\le -\cfrac{4}{\pi}\,t.$$
It yields 
$$\cfrac{1-\cfrac{1}{2\pi M(t)}}{1-\cfrac{1}{2\pi M(0)}}\le\exp\left(-\frac{2t}{\pi^2}\right).$$
In particular $$ 1-\cfrac{1}{2\pi M(t)}\le\exp(-2t/\pi^2),$$ which gives $$M(t)\le \frac{1}{2\pi}\cfrac{1}{1-\exp(-2t/\pi^2)}.$$
\end{proof}

\begin{remark}
Let us mention that this bound is stronger than the bound obtained in Proposition \ref{Bound Linfini} when $t$ goes to $+\infty$ but it is weaker when $t$ goes to 0. Typically the bound obtained in Proposition \ref{Bound Linfini} proves that the solution of the Dyson equation are in $L^1([0,T],L^\infty)$ for every $T>0$.
\end{remark}

\begin{remark}
Doing the exact similar proof one can obtain a similar result for the minimum of a solution of \eqref{dyson}. More precisely we can show that $$m(t)\ge \cfrac{1}{2\pi}\,\cfrac{1}{1+\left(\frac{1}{2\pi m(0)}-1\right)\exp\left(\frac{-2t}{\pi^2}\right)}\,.$$
Let us notice that this bound does not give any informations if $m(0)=0$ as we shall discuss in Section \ref{subsubsection:convlinfty}.
\end{remark}

Proposition \ref{Prop:longtimemax} implies without hypothesis on $\mu(0,.)$ a convergence in $L^p(\T)$ for $p\in[1,+\infty)$. In particular, it gives a new proof of the convergence in law towards $\mu_{\text{Unif}}$.
\begin{proposition}
Let $\mu$ be a solution to $\eqref{dyson}$ and $p\in[1,+\infty)$. Then we have the following convergence: 
$$\left|\left|\mu(t,.)-\cfrac{1}{2\pi}\right|\right|_{L^p(\T)}\underset{t\to\infty}{\longrightarrow} 0.$$
\end{proposition}
\begin{proof}
Since $M(t)$ goes to $1/(2\pi)$ exponentially fast and $\mu(t,.)$ is a probability density, formally we get that $\mu(t,.)$ will converge towards $1/(2\pi)$. 
\newline
For $0<\varepsilon\le 1/(2\pi)$ and $t>0$, let $$I(t,\varepsilon)=\left\{y\in\T;\, \mu(t,y)<\cfrac{1}{2\pi}-\varepsilon\right\}.$$
By using that for all $t\ge 0 $, $\mu(t,.)$ is a probability measure, we get: 
\begin{align*}
0&=1-\int_\T \mu(t,\theta)d\theta=\int_\T\left(\cfrac{1}{2\pi}-\mu(t,\theta)\right)d\theta\\
&=\int_{I(t,\varepsilon)}\left(\cfrac{1}{2\pi}-\mu(t,\theta)\right)d\theta+\int_{\T-I(t,\varepsilon)}\left(\cfrac{1}{2\pi}-\mu(t,\theta)\right)d\theta\\
&\ge\varepsilon \leb(I(t,\varepsilon))+\left(\cfrac{1}{2\pi}-M(t)\right)\left(2\pi-\leb(I(t,\varepsilon))\right).
\end{align*}
It yields that: 
\begin{equation} 
\label{lebesguemeasure}
\cfrac{2\pi\left(M(t)-\cfrac{1}{2\pi}\right)}{\varepsilon+M(t)-\cfrac{1}{2\pi}}\ge \leb(I(t,\varepsilon)).
\end{equation}
Thanks to this bound on $\leb(I(t,\varepsilon))$, we can look at the convergence in $L^p$ norms for $p\in[1,+\infty)$. Indeed: 
\begin{align}
\int_\T\left|\cfrac{1}{2\pi}-\mu(t,\theta)\right|^p d\theta=&\int_{I(t,\varepsilon)}\left|\cfrac{1}{2\pi}-\mu(t,\theta)\right|^p d\theta+\int_{\T-I(t,\varepsilon)}\left|\cfrac{1}{2\pi}-\mu(t,\theta)\right|^p d\theta\\
&\le\cfrac{1}{(2\pi)^p}\leb(I(t,\varepsilon))+2\pi\max\left(\varepsilon,M(t)-\cfrac{1}{2\pi}\right)^p
\label{convequi lp}
\end{align}
We take $\varepsilon$ depending on the time also. We choose $\varepsilon(t)^2=M(t)-1/(2\pi)$ which goes to 0 exponentially fast by Proposition \ref{Prop:longtimemax}. This choice of $\varepsilon(t)$ implies that $$\leb(I(t,\varepsilon(t)))\underset{t\to\infty}{\longrightarrow}0$$ by \eqref{lebesguemeasure}. By using \eqref{convequi lp}, we get the result.
\end{proof}

\subsubsection{Convergence of the solution in $L^\infty$ if the initial condition is positive}
\label{subsubsection:convlinfty}
As explained in Section \ref{subsubsection:maxconv} we can obtain the convergence of a solution of \eqref{dyson} if $m(0)>0$ using the Costantin-Vicol lemma to obtain a priori estimates. We present here another approach to obtain this result. We shall directly bound the amplitude (defined as the $M(t)-m(t)$) of a solution of \eqref{dyson} using other classical inequalities satisfied by the operator $A_0$. This approach is
inspired by the work done in \cite{kiselev2022} which can be used for more general equations which involves the operator $A_0$ \cite{cordoba2004maximum,constantin2012nonlinear}.

\begin{proposition}
\label{prop:amplitude}
Let $\mu$ be a smooth solution to $\eqref{dyson}$, then $$\mathcal V(t):=M(t)-m(t)\le (M(0)-m(0))\exp(-\sigma_0 t),$$
where $\sigma_0=4m(0)$.
\end{proposition}
\begin{proof}
First we state a lemma based on the properties of the operator $A_0$ used in \cite{kiselev2022}.
\begin{lemma}
\label{bound}
Let $f\in C(\T)$ such that $f(x_0)=M$ is the maximum of $f$ and $f(\tilde x_0)=m$ is the minimum of $f$ and let $$\bar f=\cfrac{1}{2\pi}\int_\T f(x)dx.$$ Then we have the following bounds: 
\begin{align*}
A_0 f(x_0)\ge&2(M-m)\cotan\left(\cfrac{\pi(\bar f-m)}{2(M-m)}\right)\\
A_0 f(\tilde x_0)\le-&2(M-m)\cotan\left(\cfrac{\pi(M-\bar f)}{2(M-m)}\right),
\end{align*}
with strict inequality if $f$ is not constant.
\end{lemma}
\begin{proof}
We just prove the first inequality, the second one can be obtained by changing $f$ in $-f$. Moreover we suppose that $f$ is not constant, otherwise the result is obvious. 
\newline
For $d\in (0,2\pi)$, we introduce the function  $f_d$ defined on $\T$ defined by $f_d(x)=M$ if $|x-x_0|\le d$ and $f_d(x)=m$ otherwise. 
We want to minimize the quantity $A_0 f(x_0)$ on the periodic continuous functions that have the same maximum, minimum and mean value than $f$. Intuitively, since $1/\sin^2(x/2)$ become infinite in 0, we have to minimize the numerator in $A_0$.
\newline
So, we look at: 
$$A_0(f-f_d)(x_0)=\int_{[-d,d]}\cfrac{f-f_d(x_0)-(f-f_d)(x_0+z)}{2\sin^2(z/2)}dz+\int_{\T-[-d,d]}\cfrac{f-f_d(x_0)-(f-f_d)(x_0+z)}{2\sin^2(z/2)}dz.$$
By definition of $f_d$ we have: $$A_0(f-f_d)(x_0)=\int_{[-d,d]}\cfrac{M-f(x_0+z)}{2\sin^2(z/2)}dz+\int_{\T-[-d,d]}\cfrac{m-f(x_0+z)}{2\sin^2(z/2)}dz.$$
By using the monotonicity of $\sin^2(z/2)$ on $[-\pi,\pi]$, we deduce that: $$A_0(f-f_d)(x_0)\ge \cfrac{1}{2\sin^2(d/2)}\left[2d M-\int_{[-d,d]} f(x)dx+2(\pi-d)m-\int_{\T-[-d,d]}f(x)dx\right].$$
Hence it remains to chose the unique $d$ such that the right term is equal to 0. It leads to $$d=\pi\cfrac{\bar f-m}{M-m}.$$
Hence we obtain for this $d$: $$A_0(f)(x_0)\ge A_0(f_d)(x_0)=2(M-m)\cotan\left(\cfrac{\pi(\bar f-m)}{2(M-m)}\right).$$
We observe that since $f$ is not constant the previous inequality is actually strict.
\end{proof}
Now we prove Proposition \ref{prop:amplitude}.
We evaluate the equation $\eqref{dysonextend}$ in $(t,x(t))$ where $M(t)=\mu(t,x(t))$. 
It gives that:
\begin{equation}
\label{majoration}
M'(t)\le-\mu(t,x(t))A_0[\mu(t,.)](x(t))\le -m(0)A_0[\mu(t,.)](x(t)).
\end{equation}
Then we use the estimate of Lemma $\ref{bound}$: $$M'(t)\le-2m(0)(M(t)-m(t))\cotan\left(\cfrac{\pi(\bar \mu(0)-m(t))}{2(M(t)-m(t))}\right).$$
By the same argument we get that: $$m'(t)\ge 2m(0)(M(t)-m(t))\cotan\left(\cfrac{\pi(M(t)-\bar \mu(0))}{2(M(t)-m(t))}\right).$$
So, we have that:
\begin{align*} 
\mathcal V'(t)\le& -2m(0)(M(t)-m(t))\left[\cotan\left(\cfrac{\pi(\bar \mu(0)-m(t))}{2(M(t)-m(t))}\right)+\cotan\left(\cfrac{\pi(M(t)-\bar \mu(0))}{2(M(t)-m(t))}\right)\right]\\
\le&-4m(0)\mathcal V(t).
\end{align*}
because of the fact: $$\cfrac{\pi(M(t)-\bar \mu(0))}{2(M(t)-m(t))}+\cfrac{\pi(\bar \mu(0)-m(t))}{2(M(t)-m(t))}=\cfrac{\pi}{2},$$ and $$\cotan(x)+\cotan(\cfrac{\pi}{2}-x)=\cotan(x)+\cfrac{1}{\cotan(x)}\ge 2.$$
Hence by Gronwall's lemma we get the result.
\end{proof}

We directly obtain from this result the convergence towards the equilibrium measure in the case $m(0)>0$. 
\begin{proposition}
\label{expconv1}
Let $\mu$ be a smooth solution to $\eqref{dysonextend}$, then there exists a finite constant $C_0$ that only depends on $\mu(0,.)$ such that for all time $t\ge 0$: $$ \left|\left|\,\mu(t,.)-\cfrac{1}{2\pi} \,\right|\right|_\infty\le C_0\exp(-\sigma_0 t),$$ with $\sigma_0=4m(0)$. 
\end{proposition} 
We can actually improve this theorem by proving that we can find a $\sigma_0$ that does not depend on $\mu(0)$ by bootstrapping the result we get. 
\begin{proposition}
\label{expconv2}
Let $\mu$ be a smooth solution to $\eqref{dyson}$, then there exists a universal constant $\sigma>0$ and a constant $\tilde C_0\in[0,+\infty]$ which is finite if and only if $m(0)>0$ such that for all time $t\ge 0$: $$ \left|\left|\,\mu(t,.)-\cfrac{1}{2\pi} \,\right|\right|_\infty\le\tilde C_0\exp(-\sigma t).$$ 
\end{proposition}
\begin{proof}
To improve our bound we start from $\eqref{majoration}$, and now instead of lower bounding $\mu(t,x(t))$ from by $m(0)$ we can bound it from below by $$\mu(t,x(t))\ge m(t)\ge \cfrac{1}{2\pi}-C_0\exp(-\sigma_0 t), $$ thanks to Proposition $\ref{prop:amplitude}$.
Using this bound and doing the same computations we get: 
\begin{align*} 
\mathcal V'(t)\le& -\left(\cfrac{1}{2\pi}-C_0\exp(-\sigma_0 t)\right)2(M(t)-m(t))\left[\cotan\left(\cfrac{\pi(\bar \mu(0)-m(t))}{2(M(t)-m(t))}\right)+\cotan\left(\cfrac{\pi(M(t)-\bar \mu(0))}{2(M(t)-m(t))}\right)\right]\\
\le&-4\left(\cfrac{1}{2\pi}-C_0\exp(-\sigma_0 t)\right)\mathcal V(t).
\end{align*}
Hence, by Gronwall lemma we have the following bound: 
$$\mathcal V(t)\le \mathcal V(0)\exp\left(-\cfrac{2t}{\pi}+4\cfrac{C_0}{\sigma_0}(1-\exp(-\sigma_0t)\right)\le \mathcal V(0)\exp\left(4\cfrac{C_0}{\sigma_0}\right)\exp\left(-\cfrac{2t}{\pi}\right).$$
This bound concludes the proof.
\end{proof}

To summarize the results about the convergence towards equilibrium obtained in this section and the previous one: we proved that if there exists a time $t^*\ge 0$ such that $m(t^*)>0$ then a solution to the Dyson equation $\mu(t,.)$ converges in $L^\infty(\T)$ to $1/(2\pi)$ exponentially fast and if for all time $t\ge 0$, $m(t)=0$ then we have at least that $\mu(t,.)$ converges to $1/(2\pi)$ exponentially fast in $L^p(\T)$ for all $p\in [1,+\infty)$. 
\newline
We expect that $\mu(t,.)$ always converges in $L^\infty(\T)$ to $1/(2\pi)$ (and so we are always in the first previous case actually). We did not succeed to prove this result. We think that we lack information to obtain this result. If we could get a similar bound as obtain Proposition \ref{prop: estimée infini} but for $C^{\alpha}(\T)$ with $\alpha>0$ instead of $L^\infty(\T)$ we would have the result. 
\subsubsection{Regularity inside the support}
To conclude this section, let us mention how to obtain bounds about the derivative of a solution to the Dyson equation that can be used to prove regularity of a solution of the Dyson equation inside its support. Indeed, as in \cite{kiselev2022} we shall prove the exponentially fast convergence to 0 of the uniform norm of the derivative of a solution to the Dyson equation when $t$ goes to $\infty$ if $m(0)>0$. We give the proof of this result to point out some bounds on the Hilbert transform and the half Laplacian that can be used in the study of equations of this type even though these lemmas are known in the literature \cite{constantin2012nonlinear,kiselev2022}.
\begin{proposition}
\label{expconvderivative}
Let $\mu$ be a smooth solution to $\eqref{dyson}$ such that $m(0)>0$. Then $$||\partial_x \mu(t,.)||_\infty\le C_1\exp(-\sigma_1 t),$$ where $\sigma_1=m(0)$ and $C_1$ is a finite constant that depends on $\mu(0,.)$. 
\end{proposition}
\begin{proof}
First, we state two lemmas about the operator $A_0$ and the Hilbert transform used in \cite{constantin2012nonlinear,kiselev2022}.
\begin{lemma}[Enhanced maximum principle]
\label{lemma: max principle der} Let $f\in C^1(\T)$ be a function with amplitude $\mathcal V=\max_\T f-\min_\T f$ and $g=f'$. Let $x_0$ be a point where $\max_\T g$ is attained. Then we have: 
\begin{equation}
\begin{split}
4 \mathcal V A_0[g](x_0)\ge& g(x_0)^2\\
A_0[g](x_0)\ge&\max(g(x_0)-\mathcal V,0)
\end{split}
\end{equation}
\end{lemma}

\begin{proof}
If $\mathcal V=0$, then $f$ is constant and $g=0$. The two results are obvious.
So we now deal with the case $\mathcal V>0$ which implies $g(x_0)>0$ (indeed otherwise $g\le 0$ and so $u$ would be a non increasing periodic function which is only possible if $f$ is constant).
\newline
We integrate by part to link $A_0[g]$ with $f$. To do this properly, we introduce a smooth increasing function defined in $\R^+$ such that $\chi(r)=0$ if $r<1$ and $\chi(r)=1$ if $r>2$ and $\chi'(y)\le 2$. We also define $\chi_R(x)=\chi(x/R)$. Now, we can lower bound properly $A_0[g]$. 
\begin{align*}
 \cfrac{A_0[g](x_0)}{2}\ge& \int_\R\chi_R(|y|)\,\cfrac{g(x_0)-g(x_0+y)}{y^2}dy\\
\ge& g(x_0)\int_{|y|>2R}\cfrac{1}{y^2} \,dy-\int_{|y|>R}\chi_R(|y|)\,\cfrac{g(x_0+y)}{y^2}dy\\
\ge& \cfrac{g(x_0)}{R}-\int_{|y|>R}\chi_R(|y|)\,\cfrac{\partial_y(f(x_0+y)-f(x_0))}{y^2}dy\\
=&\cfrac{g(x_0)}{R}+\int_{|y|>R}(f(x_0+y)-f(x_0))\,\partial_y\left(\cfrac{\chi_R(|y|)}{y^2}\right)dy\\
=&\cfrac{g(x_0)}{R}+\int_{y>R}(f(x_0+y)-f(x_0-y))\,\partial_y\left(\cfrac{\chi_R(y)}{y^2}\right)dy\\
\ge&\cfrac{g(x_0)}{R}-\cfrac{2\mathcal V}{R^2}, 
\end{align*}
because of the fact that: 
\begin{align*}
\int_{y>R}(f(x_0+y)-f(x_0-y))\,\partial_y\left(\cfrac{\chi_R(y)}{y^2}\right)dy&\ge -\mathcal V\int_{y>R},\cfrac{\chi'_R(y)}{Ry^2}+\cfrac{2\chi_R(y)}{y^3}\,dy\\
&\ge-\mathcal V\left(\int_R^{2R}\cfrac{2}{Ry^2}\,dy+\int_R^{+\infty}\cfrac{2}{y^3} \right)\,dy=-\cfrac{2\mathcal V}{R^2}.
\end{align*}
We conclude by choosing $R=4\mathcal V/g(x_0)$ (obtained by optimizing in $R$ the previous bound) and $R=2$.

\end{proof}

\begin{lemma}
\label{lemma: bound Hilbert trans}
With same assumptions of Lemma $\ref{lemma: max principle der}$ we have: $$H[g](x_0)^2\le 32 \,\mathcal V A_0[g](x_0).$$
\end{lemma}
\begin{proof}
We cut $H[g](x_0)$ in two parts to do an integration by part far from the singularity. 
$$H[g](x_0)=2\int_{-R}^R\cfrac{g(x_0)-g(x_0+y)}{y}\, dy+2\int_{R}^{+\infty}\cfrac{g(x_0+y)-g(x_0-y)}{y}\, dy.$$The first term can be compared with the half Laplacian directly: $$2\int_{-R}^R\cfrac{g(x_0)-g(x_0+y)}{|y|}=2\int_{-R}^R\cfrac{g(x_0)-g(x_0+y)}{y^2}|y|\,dy\le R A_0[g](x_0).$$
For the second, integrating by part yields: 
$$\left|\left[\cfrac{2f(x_0)-f(x_0+y)-f(x_0+y)}{y}\right]_R^{+\infty}-\int_R^{\infty}\cfrac{2f(x_0)-f(x_0+y)-f(x_0+y)}{y^2}\, dy\right|\le \cfrac{2\mathcal V}{R}+\cfrac{2\mathcal V}{R}=\cfrac{4\mathcal V}{R}.$$
So we get that: $$|H[g](x_0)|\le R A_0[v](x_0)+\cfrac{8\mathcal V}{R}.$$
We finally choose the optimal $R$ to have the smallest bound which leads to the result.
\end{proof}
Thanks to these lemmas we can now prove Proposition \ref{expconvderivative}.
\newline
We look at the equation satisfies by $\nu:=\partial_x \mu$ by differentiating $\eqref{dysonextend}$. 
$\nu$ is solution to: $$\partial_t \nu +\partial_x \nu H[\mu]+\nu H[\nu]+\nu A_0[\mu]+\mu A_0[\nu]=0.$$
By using that: $A_0[u]=\partial_x H[u]=H[\partial_x u]=H[\nu]$ we have that $\nu$ satisfies the following equation: 
\begin{equation}
\label{dysonderivée}
\partial_t \nu +\partial_x \nu H[\mu]+2\nu H[\nu]+\mu A_0[\nu]=0.
\end{equation}
We introduce $M_1(t)=\max_\T \partial_x\mu(t,.)$. By compactness and regularity of $\mu$ we can find $x(t)$ such that $M_1(t)=\partial_x\mu(t,x(t))$. We evaluate $\eqref{dysonderivée}$ in $(t,x(t))$ to obtain that $$M_1'(t)\le -2\nu(t,x(t))H[\nu(t,.)](x(t))-\mu(t,x(t))A_0[\nu(t,.)](x(t)).$$
Then, it yields that: 
\begin{align*}
-2\nu(t,x(t))H[\nu(t,.)](x(t))&\le \nu(t,x(t))^2+(H[\nu(t,.)](x(t)))^2\\
&\le 4 \mathcal V(t)A_0[\nu(t,.)](x(t))+32 \,\mathcal V(t)A_0[\nu(t,.)](x(t)),
\end{align*}
by Lemma \ref{lemma: max principle der} and Lemma \ref{lemma: bound Hilbert trans}.
Hence, we deduce that: $$M_1'(t)\le\left(-m(0)+36 \mathcal V(t)\right)A_0[\nu(t,.)](x(t)).$$
Using Lemma \ref{lemma: max principle der} and the fact that $\mathcal V(t)$ converges exponentially fast to $0$ by Theorem \ref{expconv1}, we can say that for $t$ large enough: 
\begin{align*}
M_1'(t)&\le\left(-m(0)+36 \mathcal V(t)\right)\max(M_1(t)- \mathcal V(t),0)\\
&\le\left(-m(0)+36 \mathcal V(t)\right)(M_1(t)- \mathcal V(t)).
\end{align*}
Let $K_1:=\max(36\mathcal V(0),m(0)\mathcal V(0))$ and $\sigma_1=m(0)<\sigma_0$. Thanks to Proposition $\ref{prop:amplitude}$ we get:  
$$ M_1'(t)\le (-\sigma_1+K_1\exp(-\sigma_0 t))M_1(t)+K_1\exp(-\sigma_0 t).$$
By Gronwall's lemma we conclude that there exists $C_1\ge 0$ such that $$M_1(t)\le C_1\exp(-\sigma_1 t).$$
We can do the same proof to lower bound $\min_\T\partial_x\mu(t,.)$ and obtain the result.
\end{proof}

\subsection*{Acknowledgments} The first author acknowledges partial support from the chair FDD (Institut Louis Bachelier) and the Lagrange Mathematics
and Computing Research Center. The two authors are thankful to Pierre-Louis Lions (Collège de France) for pointing out the problem.

\renewcommand{\MR}[1]{}
\bibliographystyle{smfplain}
\bibliography{ref}

\end{document}